\documentclass[11pt,thmsa]{article}%
\usepackage{amsmath}
\usepackage{amsfonts}
\usepackage{amssymb}
\usepackage{graphicx}%
\newtheorem{theorem}{Theorem}[section]

\newtheorem{corollary}[theorem]{Corollary}

\newtheorem{definition}[theorem]{Definition}

\newtheorem{lemma}[theorem]{Lemma}

\newtheorem{question}[theorem]{Question}
\newtheorem{proposition}[theorem]{Proposition}
\newtheorem{remark}[theorem]{Remark}

\newenvironment{proof}[1][Proof]{\noindent\textbf{#1.} }{\ \rule{0.5em}{0.5em}}

\begin{document}
\title{Geodesic and curvature of piecewise flat Finsler surfaces\thanks{Supported by NSFC (no. 11271198, 51535008) of China}}

\author{Ming Xu$^1$ and Shaoqiang Deng$^2$\thanks{Corresponding author} \\
\\
$^1$College of Mathematics\\
Tianjin Normal University\\
Tianjin 300387, P. R. China\\
Email: mgmgmgxu@163.com.\\
\\
$^2$School of Mathematical Sciences and LPMC\\
Nankai University\\
Tianjin 300071, P. R. China\\
E-mail: dengsq@nankai.edu.cn}

\date{}
\maketitle

\begin{abstract}
 A  piecewise flat Finsler metric  on  a triangulated surface $M$ is a metric whose restriction to  any triangle is a flat triangle in some Minkowski space with straight edges. One of the main
purposes of this work is to study the properties of geodesics on a piecewise flat Finsler surface, especially when it meets a vertex. Using the edge-crossing equation, we define two classes of piecewise flat Finsler surfaces, namely,   Landsberg type and Berwald type. We deduce an explicit condition for a geodesic to be extendable at a vertex,
 and define the curvature which measures the \textit{amount} of such extensions. The dependence of the curvature on an incoming or outgoing tangent direction corresponds to the feature of flag curvature in Finsler geometry. When the  piecewise flat Finsler surface is of Landsberg type, the curvature is only relevant to the vertex, and
we prove a combinatoric Gauss-Bonnet formula which generalizes both the Gauss-Bonnet formulas for  piecewise flat Riemannian manifolds  and for  smooth Landsberg surfaces.



\end{abstract}
\section{Introduction}

A closed manifold $M$ is called a {\it piecewise flat Riemannian manifold} if it is given a triangulation and endowed with a metric, called a {\it piecewise flat metric}, such that all the simplices are Euclidean \cite{CMS1984}.
The systematical study on
the curvature of these spaces was started by T. Regge in 1961 \cite{Regge1961}, and thoroughly explored by J. Cheeger, W. M\"{u}ller and
R. Shrader in 1984 \cite{CMS1984}. The behavior of geodesics on a piecewise flat Riemannian manifold was studied by D. Stone in 1976 \cite{St1976}. In recent years, the study of such spaces  has become rather active due to its intriguing applications to computer programming for $3$-dimensional visualization and many other fields. In addition, there are many research works on the geometric analysis for piecewise flat manifolds; see for example
\cite{CL2003, DGL2008, Gli2005-1, Gli2005-2, Gli2011}.

Let us explain the  physical motivation to consider piecewise flat Riemannian manifold and Finsler spaces, especially
their curvatures. An interesting fact is that the scalar curvature and the mean curvature of such spaces enter into the Hilbert action principle from which  the Einstein field equations
can be derived, hence the theory has application to the general relativity. Motivated by this, Regge introduced the scalar curvature of piecewise flat Riemannian manifolds and dealt with the convergence,  but he did not give a rigorous proof. The theory is called Regge's calculus by physicists. From then on, many physicists employed Regge's calculus as a tool for constructing a quantum theory of gravity as well as for establishing some other physical theories; see for example \cite{WA, WO}. Regge's calculus was made rigorous  by Cheeger,  M\"{u}ller and
Shrader in  \cite{CMS1984}.
Since recently Finsler geometry has proven to be very useful in physics and other scientific fields, we believe that a generalization of
their theory to this more generalized case will definitely find many applications in physics as well as other fields.

The idea to define {\it piecewise flat
Finsler metrics} and {\it piecewise flat Finsler manifolds} is similar to that for piecewise flat Riemannian manifolds. In fact, given  a triangulation of a manifold, we can  define the metric such that each simplex can be isometrically identified with a flat simplex in some Minkowski space. Typical examples
are the boundaries of polyhedral regions in Minkowski spaces. Notice that there is an essential difference between piecewise flat Riemannian metrics and piecewise flat Finsler metrics, since
for each  dimension, up to affine equivalence, there is a unique Euclidean metric, but there are infinitely many different Minkowski metrics.
As far as we know, there has been very few research works considering piecewise flat Finsler manifolds, even for the $2$-dimensional case. Since Finsler geometry has proven to be very useful in various scientific fields, including general relativity, biology and medical imaging, we believe that the study of
piecewise flat Finsler spaces, which are not so smooth as the regular ones
having been considered extensively in the literature, will definitely  find interesting applications in the reality world.

The goal of this paper is to study in some depth the geometric properties of piecewise flat Finsler manifolds. For the sake of simplicity and clarity,  we concentrate on the $2$-dimensional case.  The main results of this work are some explicit description of  geodesics and curvature of piecewise flat Finsler surfaces.

The behavior of geodesics crossing an edge (or more generally a simplex of codimension $1$)
is very classical. In the literature, it is called the Snell-Descartes law, which is related
to the reflection and refraction of light rays. As this is the generic situation for geodesics, i.e., locally the initial direction of any geodesic can be perturbed to avoid
passing a vertex, the behavior of geodesics crossing an edge provides the foundation for our further study. Using the edge-crossing equation for
geodesics, we define two special classes of piecewise flat Finsler surfaces,
namely, piecewise flat {\it Landsberg} surfaces and piecewise flat {\it Berwald} surfaces. Among these two concepts,
the Berwald condition is stronger, which implies the other. We will show that a piecewise flat Landsberg surface shares many important properties with a smooth Landsberg surface. For example, when $M$ is connected and $x$ is not a vertex, the length of indicatrix circle $\mathcal{S}_x^\pm M\subset T_xM$ of incoming or outgoing unit tangent vectors at $x$, with respect to the metric defined by the Hessian matrix, are all the same (we will denote this constant as $\theta_M$, see Lemma \ref{lemma-theta-M}).

The consideration of the behavior of geodesics at vertices is tricky.
The discussion can be carried out
in a small neighborhood of the vertex, which is modelled as a union of Minkowski cones with the unique common vertex. We call it the {\it tangent cone}, and simply denote it as $T_xM$. The key observation is that, a
geodesic in $T_xM$ not passing $x$ can not rotate around $x$ for infinitely many times, hence it can be presented as the union of two rays and a finite number of line segments (see Theorem \ref{lemma-geodesic-not-passing-vertex-x}). This fact is obvious for piecewise flat Riemannian surfaces, But not very easy in the Finsler context. Its proof for the general situation
is quite lengthy, but for special cases, for example when $F$
is reversible or of Landsberg type, there are some shortcuts. This key observation indicates there are only three types of maximally extended geodesics in $T_xM$, those consisting of one ray at $x$, those consisting of two rays at $x$, and
those not passing $x$ (see Corollary \ref{cor-1}).

Using above local description for geodesics near a vertex, we shall give  several equivalent conditions for a geodesic in
$T_xM$ extended from one ray at $x$ to two rays (see Theorem \ref{main-thm-1} and Theorem \ref{main-thm-2}).
It must be pointed out that, the concept of curvature naturally
appears to measure the \textit{amount} of all such extensions for a given ray at $x$.
It is a natural
generalization of curvature in the piecewise flat Riemannian context. Similarly, it vanishes
for non-vertex points. Its dependence on the (incoming or outgoing) tangent vector is similar to that of the flag curvature of a smooth Finsler surface.
But when $M$ is of Landsberg type, the curvature is not relevant to the vertex, similar to the curvature form of a smooth Landsberg surface (see Theorem \ref{main-thm-3}).

Another main goal of this paper is to present a combinatoric version of the Gauss-Bonnet-Chern formula for piecewise flat Finsler surfaces. In Finsler geometry, a similar Gauss-Bonnet formula
as in Riemannian geometry is only available for a Landsberg metric \cite{BCS2000} \cite{BCS1996}. Our combinatoric version of the formula for a compact connected
piecewise flat Landsberg surface (without boundary) can be expressed as  $$\sum K(x)=\theta_M\chi(M),$$
in which the left side is a finite sum for all vertices and $\theta_M$ is
the constant previously mentioned for connected piecewise flat Landsberg surfaces
(see Theorem \ref{main-thm-4}).

Finally, we would like to mention that many notions and methods are also valid for
high dimensional piecewise flat Finsler spaces, and we expect that a combinatoric
Gauss-Bonnet-Chern formula for piecewise flat Landsberg manifolds will  also be valid; see
\cite{BS1996} for the Finslerian case.

This work is organized as the following. In Section 2, we summarize some fundamental knowledge on Finsler geometry. In Section 3, we give the definition
and notations for a piecewise flat Finsler surface. In Section 4, we define
geodesics on a piecewise flat Finsler surface and study its behavior crossing
an edge. We use the edge-crossing equation for a geodesic to define piecewise flat Landsberg surface and piecewise flat Berwald surface. In Section 5, we study the behavior of a geodesic at a vertex.
In Section 6, we define the curvature at its vertices under the above mentioned conditions. In particular, we prove that the curvature is irrelevant to the tangent vector when $M$ is of Landsberg type. In addition,
 we answer the question when a geodesic can be extended at a vertex. In Section 7,
we prove a combinatoric
Gauss-Bonnet formula for piecewise flat Landsberg surfaces.

\section{Preliminaries}
\subsection{Minkowski norm and Finsler metric}
A {\it Minkowski norm} $F$ on a real linear space $\mathbf{V}$ is a continuous
function $F:\mathbf{V}\rightarrow[0,+\infty)$ satisfying
\begin{description}
\item{\rm (1)} $F$ is positive and smooth when restricted to $\mathbf{V}\backslash \{0\}$.
\item{\rm (2)} $F(\lambda y)=\lambda F(y)$ when $\lambda\geq 0$.
\item{\rm (3)} For any linear coordinates $y=y^ie_i\in\mathbf{V}$, the Hessian matrix $$(g_{ij}(y))=\frac12[F^2]_{y^iy^j}$$
is positive definite for any $y\in\mathbf{V}\backslash \{0\}$.
\end{description}

The Hessian matrix of  the second derivatives of $F^2$ defines an inner product
$$\langle u,v\rangle_y^F=\frac12\frac{\partial^2}{\partial s\partial t}F^2(y+tu+sv)|_{s=t=0}=
g_{ij}(y)u^iv^j,$$
where $u=u^ie_i$ and $v=v^je_j$. The third derivative of $F^2$ at $y\in\mathbf{V}\backslash \{0\}$ defines the
Cartan tensor
$$C^F_y(u,v,w)=\frac14\frac{d^3}{drdsdt}F^2(y+ru+sv+tw)|_{r=s=t=0}
=\frac12\frac{d}{dt}\langle u,v\rangle^F_{y+tw}|_{t=0}.$$
Notice that the Cartan tensor $C^F_y(u,v,w)$ is symmetric with respect to $u$, $v$ and $w$, and  $C^F_y(u,v,y)=0$ for any $u$ and $v$ in $\mathbf{V}$.

A {\it Finsler metric} $F$ on a smooth manifold $M$ is a continuous function
$F:TM\rightarrow[0,+\infty)$ such that it is smooth on the slit
tangent bundle $TM\backslash 0$, and its restriction to each tangent space
is a Minkowski norm. We also call $(M,F)$ a Finsler space.

A Minkowski norm (resp., a Finsler metric) $F$ is {\it reversible} if $F(y)=F(-y)$
for any $y\in\mathbf{V}$ (resp., $F(x,y)=F(x,-y)$ for any $x\in M$ and $y\in T_xM$).

 As  examples, Riemannian metrics are the the special class of Finsler metrics such that Cartan tensors vanish everywhere, or equivalently, for any local coordinates the Hessian $(g_{ij}(x,y))$ only depends
on $x$. Randers metrics are the most simple and important non-Riemannian Finsler metrics \cite{Ra41}, which are of the form $F=\alpha+\beta$ where $\alpha$ is
a Riemannian metric and $\beta$ is a 1-form. They can be further generalized
to $(\alpha,\beta)$-metrics which are of the form $F=\alpha\phi(\beta/\alpha)$,
where $\alpha$ and $\beta$ are similar as for Randers metrics and $\phi$ is
a smooth function. Recently, there are many research works on $(\alpha,\beta)$-spaces; see for example \cite{BCS07}.

\subsection{Minkowski space, triangle and cone}

A real
linear space $\mathbb{R}^n$ can be viewed as an Abelian Lie group as well as its Lie algebra. So a Minkowski
norm $F$ on $\mathbb{R}^n$ determines a bi-invariant Finsler metric on it as well, which will be simply denoted as the same $F$. We call $F$ a Minkowski metric and $(\mathbb{R}^n,F)$ a {\it Minkowski space}, or a {\it Minkowski plane} when $n=2$.
Minkowski spaces are the most standard and simple flat Finsler spaces, i.e., the   flag curvature is identically zero. The geodesics of Minkowski spaces are straight lines.

A Finsler
space $(T,F)$ (or its completion) is called a {\it Minkowski triangle}, if it is isometric to a (closed or open) triangle in the
Minkowski plane $(\mathbb{R}^2,F)$, which have three vertices and three edges.
Similarly we can define a ($2$-dimensional) {\it Minkowski cone}, corresponding to the region between two rays from a point in $(\mathbb{R}^2,F)$. So we see a Minkowski cone has
one vertex and two edges.
Roughly speaking, when the "angle" between the two edges is not too big, a Minkowski cone can be regarded as a Minkowski triangle with two vertices (and also the edge between them) at infinity.
In later discussions, we also need to deal with some important Minkowski cones
with big "angles".
For example, with the vertices and edges properly chosen, any whole Minkowski space and any Minkowski half plane can be regarded as Minkowski cones.

\subsection{Geodesic and curvature}

In a Finsler space $(M,F)$, the length of a tangent vector is defined  by
the metric function $F$. So for any piecewise smooth path $\{p(t),t\in[a,b]\}$
(sometimes, we may denote it as $\{p(t),x_1,x_2\}$, where $x_1$ and $x_2$ are the initial point and end point respectively), the arc length
\begin{equation}\label{define-arc-length}
l(\{p(t),t\in[a,b]\})=\int_a^b F(p(t),\dot{p}(t))dt,
\end{equation}
 and the distance function
\begin{eqnarray}\label{define-distance}
d(x_1,x_2)=\inf\{l(\{p(t),t\in[a,b]\}),\forall \mbox{ piecewise smooth path } \{p(t),x_1,x_2\}\}
\end{eqnarray}
can be defined. It should be mentioned that the distance function is reversible iff the metric $F$ is reversible, i.e.
$F(x,y)=F(x,-y)$ for any $x\in M$ and $y\in T_xM$.

Using the variational method for the arc length functional (an equivalent and more convenient approach is to use the energy functional instead of the length functional), we can define the geodesics on a Finsler space, which  are smooth curves $\{c(t), t\in I\}\subset M$, where $I$ is a nonempty interval,
satisfying the local minimizing principle, i.e.,  for each $t_0\in I$, there
exists a sufficiently small positive $\epsilon$, such that for any $t_1,t_2\in
I\cap(t_0-\epsilon,t_0+\epsilon)$, $t_1<t_2$, we have $l(\{c(t),t\in[t_1,t_2]\})=d(c(t_1),c(t_2))$.

Flag curvature is a natural generalization of sectional curvature in
Riemannian geometry.  Given a  nonzero vector $y\in T_xM$ (the flag pole),
and a tangent plane $\mathbf{P}\subset T_xM$ spanned by $y$ and $w$, the flag curvature is defined as
\begin{equation}\label{5999}
K^F(x,y,\mathbf{P})=K^F(x,y,y\wedge w)=\frac{\langle R^F_y w,w\rangle^F_y}{\langle w,w\rangle^F_y\langle y,y\rangle^F_y-[{\langle y,w\rangle^F_y}]^2},
\end{equation}
where $R^F_y:T_xM\rightarrow T_xM$ is the \textit{Riemann curvature} which appear in the Jacobi equation for a smooth
family of geodesics of constant speeds (see \cite{BCS2000} for its explicit presenting by standard local coordinates).

Notice that the flag curvature  depends on $y$ and $\mathbf{P}$ rather than $w$. When $F$ is a Riemannian metric, it is just the sectional curvature and it is  irrelevant to the choice
of the flag pole.

\section{The definition and some notations}

Let $M$ be a surface endowed with a locally finite triangular decomposition, i.e.,
$M=\bigcup T_i$, in which each $T_i$ is a closed triangle in $M$, and
any different $T_i$ and $T_j$ may only intersect at exactly
one common edge or one common vertex.

\begin{definition}\label{define-PL-Finsler-surface}
We say that $M$ is a {\it piecewise flat Finsler surface}, if each triangle $T_i$ is endowed
with a flat metric $F_i$, such that
\begin{description}
\item{\rm (1)} Each $(T_i,F_i)$ is a Minkowski triangle;
\item{\rm (2)} If $T_i$ and $T_j$ have a common edge, then $F_i$ and $F_j$ coincide on this edge.
\end{description}
\end{definition}

For any $x\in M$, there are only finitely many closed triangles containing it. We list them as $(T_1,F_1)$, $\ldots$, $(T_n,F_n)$.
They can be enlarged to Minkowski
cones $C_{x,1}$, $\ldots$, $C_{x,n}$ with the unique common vertex $x$, and each
$F_i$ can be naturally extended to $C_{x,i}$.  We call
$T_xM=\mathop{\bigcup}\limits_{i=1}^n C_{x,i}$ a {\it tangent cone}.
The point $x$ is the only {\it vertex} in $T_xM$. When $i\neq j$, $C_{x,i}\cap C_{x,j}$ is either $x$, or a ray at $x$ (we call it an {\it edge}).

For example, when $x$ is not an edge point or a vertex, $T_xM$ contains only one single Minkowski cone which can be identified with a Minkowski plane. When $x$ is an edge point but not a vertex, $T_xM$ contains two half Minkowski planes with the same boundary line, each of which is a Minkowski cone.
When $x$ is a vertex, because we have required any two different triangles can
have at most 1 common edge, there are at least 3 triangles contains $x$, i.e. $T_xM$ contains at least 3 Minkowski cones. Notice this requirement is not
essential, but convenient. It can always be achieved by suitable subdivisions
for $M$, which does not change the geometry of the piecewise flat surface.

Notice for each $x\in M$, it has a neighborhood
which can be isometrically imbedded into $T_xM$, so it is convenient to use
$T_xM$ instead of $M$ to study the local geometric properties (behavior of geodesics, curvatures, etc.) around $x$.

Now we define tangent vectors. The one-side derivative of the ray $\{c(t), t\in[0,\infty)\}\subset C_{x,i}\subset T_xM$ at $c(0)=x$ defines an {\it outgoing tangent vector} at $x$ in $C_{x,i}$, denoted as $\dot{c}(0)^+$ or $v^+$. It can be presented as an \textit{arrow} in $C_{x,i}$
 with the initial point $x$. Similarly, we can define an {\it incoming tangent vector} $\dot{c}(0)^-$ or $v^-$ at $x$ in $C_{x,i}$, for the ray $\{c(t),t\in(-\infty,0]\}$ with $c(0)=x$, and present it as an \textit{arrow} in $C_{x,i}\subset T_xM$ with end point  $x$. To summarize, in both cases, we call it a \textit{tangent vector} at $x$ in $C_{x,i}$, and  use the superscript $\pm$ to indicate it is incoming or outgoing.

Positive scalar multiplications on $T_xM$, regarding $x$ as the origin, are well defined on each $C_{x,i}$, preserving the set of all outgoing (incoming) tangent vectors at $x$, and changing $F_i$ by a scalar. Negative scalar multiplications of an incoming (outgoing) tangent vector at $x$ can be naturally viewed as an outgoing (incoming) tangent vector  at $x$, with its initial and end points of the arrow switched.
Generally speaking, other addition and substraction on $T_xM$ are only conditionally defined.

The length of a tangent vector $v^\pm$ at $x$ is given by the norm $F_i$, when
$v^\pm$ is presented as an arrow in $C_{x,i}$.
If a nonzero tangent vector $v^\pm$ belongs to the edge between two different cones $C_{x,i}$ and $C_{x,j}$, by (2) in Definition \ref{define-PL-Finsler-surface}, $F_i$
and $F_j$ defines the same length for $v^\pm$ in this case.
If the cone $C_{x,i}$ is specified, we will simply denote the length of $v^\pm$ as $F(v^\pm)$.

For the simplicity in later discussions, we will use the following notations for line segments and rays. For $x_1$ and $x_2$ from the same Minkowski space, cone or triangle, we denote by $L_{x_1,x_2}$ the path along the straight line segment from $x_1$ to $x_2$, by $v_{x_1,x_2}$ the vector represented by the arrow from $x_1$ to $x_2$ (it can be naturally regarded as an outgoing tangent vector at $x_1$ as well as an incoming tangent vector at $x_2$). Moreover, we denote by $R_{x',v^\pm}$ the curve $c(t)$ along the ray at $c(t')=x'$, which is defined for either $t\in(-\infty,t']$ or $t\in[t_0,\infty)$ for some $t'$, depending on the outgoing or incoming vector $v^\pm$.

The set of all unit outgoing (incoming) tangent vectors in $C_{x,i}$ is denoted as $\mathcal{S}_{x,i}^+M$ ($\mathcal{S}_{x,i}^-M$ respectively). The unions $\mathcal{S}_x^\pm M$ of $\mathcal{S}_{x,i}^\pm M$ are called respectively the {\it outgoing indicatrix} and the {\it incoming indicatrix} at $x$. They are circles on which the arc length functionals are defined according to the Hessian matrices. We will simply use $l_x^\pm(\cdot)$ to denote arc length functional on $\mathcal{S}_x^\pm M$, and use it to define angles. Notice $l_x^+(\mathcal{S}_x^+M)=l_x^-(\mathcal{S}_x^-M)$ is obviously true when $x$ is not edge or vertex points, but generally it is not true when $x$ is on an edge, and in particular, when $x$ is
a vertex.

Now we define the \textit{angle}, which has different appearances
and notations in the following three cases.

{\bf Case 1.}
Let $\{v^\pm(t),t\in I\}$ be a continuous monotonous family of vectors in $\mathcal{S}_x^\pm M$, that is,  $v^\pm(t)$ keeps rotating along $\mathcal{S}_x^\pm M$ in the same direction. Then we define $l_x^\pm(\{v^\pm(t),t\in I\})$ to be the angle that this family of unit vectors in $\mathcal{S}_x^\pm M$ have swiped.

{\bf Case 2.} As a preliminary, we define the projection maps
$\mathrm{Pr}_x^\pm:T_xM\backslash\{x\}\rightarrow
\mathcal{S}_x^\pm M$.
Given $x'\neq x$ in $T_xM$, there exist unique unit
tangent vectors
$$\mathrm{Pr}_x^+(x')=\frac{v_{x,x'}}{F(v_{x,x'})}
\in\mathcal{S}_x^+ M$$
and
$$\mathrm{Pr}_x^-(x')=\frac{v_{x',x}}{F(v_{x',x})}
\in\mathcal{S}_x^- M.$$

For a curve $\{x(t),t\in I\}$ on $T_xM$ not passing $x$ and $\mathrm{Pr}_x^\pm(\{x(t),t\in I\})$ rotates monotonously in $\mathcal{S}_x^\pm M$,  we define
\begin{eqnarray*}
\sphericalangle_x^+(\{x(t),t\in I\})&=&l_x^+(\{\mathrm{Pr}_x^+(p(t)),t\in I\})\\
&=&l_x^+(\{\frac{v_{x,x(t)}}{F(v_{x,x(t)})}
\in\mathcal{S}_x^+M, t\in I\})
\end{eqnarray*}
to be the angle that this curve has swiped in $\mathcal{S}_x^\pm M$. The definition of $\sphericalangle_x^-(\cdot)$ is similar, using the projection map $\mathrm{Pr}_x^-$.

{\bf Case 3.}
We can also define the angle between two outgoing (or incoming) tangent vectors at $x$ as in classical geometry.

If  $v_1$ and $v_2$ are two nonzero tangent vectors satisfying the following conditions:
\begin{description}
\item{\rm (1)} They can be parallelly moved to be outgoing
    tangent vectors $v_{x,x_1}$ and $v_{x,x_2}$ with $x_1$ and $x_2$ in some $C_{x,i}\subset T_xM$.
\item{\rm (2)} The cone $C_{x,i}\subset T_xM$ described in (1) is unique.
\end{description}
Then up to re-parametrization, there exists a unique monotonous
smooth family of unit vectors $\{v(s),s\in I\}\subset
\mathcal{S}_x^+M$ connecting $v_{x,x_1}/F_i(v_{x,x_1})$
with $v_{x,x_2}/F_i(v_{x,x_2})$, and we define,
$$\angle_x^+(v_1,v_2)=l_x^+(\{v(s),s\in I\}).$$
We may also take the following equivalent definition,
$$\angle_x^+(v_1,v_2)=\sphericalangle_x^+(L_{x_1,x_2}).$$

The definition for $\angle_x^-(\cdot,\cdot)$ is similar.

Notice when $x$ is not an edge point or a vertex,
$\angle_x^{\pm}(v_1,v_2)$ can be defined when $v_1$ and $v_2$ are not vectors in opposite directions. We can use a suitable line
passing $x$ to divide the Minkowski plane $T_xM$ into two half
plane such that one of them contains both $v_1$ and $v_2$ (after they have been parallelly moved to be outgoing or incoming tangent vectors at $x$), and treat $x$ as an edge point. Roughly speaking, we take the \textit{less than flat angle}.

When $x$ is an edge point but not a vertex, and $v_1$ and $v_2$ are parallel to the edge with opposite directions, the angle
$\angle_x^\pm(v_1,v_2)$ may not be well defined because different \textit{flat angles} may have different values.

When $x$ is a vertex, because we have assumed there are at least three different $C_{x,i}$'s in $T_xM$, for any  $x_1\neq x\neq x_2$ in $C_{x,i}$, $\angle_x^+(v_{x,x_1},v_{x,x_2})$
and $\angle_x^-(v_{x_1,x},v_{x_2,x})$ are well defined.

\section{Geodesics and the edge-crossing equation}
\subsection{Geodesics of a piecewise flat Finsler surface}

On a piecewise flat Finsler surface $(M,F)$, we can similarly define piecewise smooth curves or paths locally as a finite union of smooth curves defined on closed intervals. A smooth curve is required to be contained in a single
Minkowski triangle. Notice that the tangent vector field $\dot{c}(t)$ on a piecewise smooth curve $c(t)$ are
well-defined
almost everywhere. The one-side tangent vectors $\dot{c}^\pm(t)$ at edge or vertex points are viewed as an outgoing and incoming tangent vector respectively. So we have the arc length $l(\{c(t),t\in I\})$ as in (\ref{define-arc-length}), and the distance function $d(\cdot,\cdot)$ as in (\ref{define-distance}).

Notice that if we change the term piecewise smooth in the above notion
to piecewise linear, or if we change the triangulation by subdivisions, the distance function $d(\cdot,\cdot)$ will not be changed.

Now we are ready to define  geodesics  on a piecewise flat Finsler surface.

\begin{definition}
A geodesic on a piecewise flat Finsler surface $M$ is a piecewise smooth curve $\{c(t),t\in I\}\subset M$ defined on some interval $I$, satisfying  the constant speed and local minimizing principles as the following:
\begin{description}
\item{\rm (1)} The length of $\dot{c}(t)$ is a nonzero constant wherever it is defined.
\item{\rm (2)} For any $t_0\in I$, there exists an $\varepsilon>0$, such that
 whenever $t_1,t_2\in I\cap(t_0-\varepsilon,t_0+\varepsilon)$ and $t_1<t_2$, $\{c(t),t\in[t_1,t_2]\}$ is a minimizing path from $c(t_1)$ to $c(t_2)$,
 i.e., $l(\{c(t),t\in[t_1,t_2]\})=d(c(t_1),c(t_2))$.
\end{description}
\end{definition}

 In order to get explicit descriptions for geodesics on a piecewise flat Finsler surface, we  study their local behaviors in the rest of this section
for those crossing an edge, and in Section 5 for those through  a vertex.

The behavior of a geodesic inside an open triangle is easy. It must be a straight line segment. It can be further extended in both directions until it meets an edge or a vertex.

The behavior of a geodesic crossing an  edge at some point
 $x$ which is not a vertex, is described by a classical theory called the
Snell-Descartes law. To make this work self-contained,
we will present this law  in two approaches below, the variational method
and the convexity technique.

\subsection{The edge-crossing equation of geodesics}

Let $x$ be an edge point but not a vertex on the piecewise flat
Finsler surface $(M,F)$, and $\{c(t),t\in I\}$ with $x=c(0)$ is a unit speed geodesic.

First we assume $x=c(0)$ and the interval $I$ contains $(-\varepsilon,\varepsilon)$
for some $\varepsilon>0$. Without loss of generality, we assume the speed of the geodesic $c(t)$ is $1$. By the local minimizing property of geodesics, there are only two cases.

In the first case, there exists  a sub-interval $I_1\subset I$ such that the geodesic segment $\{c(t), t\in I_1\}$, is  contained in an edge. Then it can be extended to the whole edge.

In the second case, none of the two tangent vectors $\dot{c}^+(0)$ and $\dot{c}^-(0)$ at $x$, outgoing and incoming respectively, coincides the directions of the edge $E$. Then the the geodesic contains two line segments in two sides of $E$ with the common end point $x$.

To study the local behavior of the geodesic around $x$.
Now we consider the geodesic $c(t)$ in the tangent cone $T_xM$.
Denote $(C_{x,1},F_1)$ and $(C_{x,2},F_2)$ the two Minkowski half planes in $T_xM$, with the straight line $E=C_{x,1}\cap C_{x,2}$ passing $x$. Assume the geodesic $c(t)$ goes from $C_{x,1}$ to $C_{x,2}$. Then in $T_xM$, this geodesic can be extended to $t\in (-\infty,\infty)$, i.e.,  it is the union of the two rays $R_{x,\dot{c}(0)^-}\subset C_{x,1}$ and $R_{x,\dot{c}(0)^+}\subset C_{x,2}$. Let $v$ be a $F$-unit vector in $E$.

Notice that such a geodesic is globally minimizing. In fact we have the following lemma.

\begin{lemma} \label{lemma-global-minimizing} Let $x$ be any point on
a piecewise flat Finsler surface.
Assume $\{c(t),t\in(-\infty,\infty)\}$ is a geodesic on
$T_xM$ consisting of two rays at $x$. Then it is a minimizing geodesic, i.e. for any  $t_1<t_2$, it is a minimizing path from $x_1=c(t_1)$ to $x_2=c(t_2)$.
\end{lemma}
\begin{proof} Notice that the geodesic $\{c(t),t\in(-\infty,\infty)\}$ indicated in the lemma is preserved
by any positive scalar multiplication on $T_xM$, which preserves each ray at $x$, and changes the arc length and distance by the same scalar. By a suitable positive scalar multiplication, any $x_1$
and $x_2$ on this geodesic can be moved to be sufficiently close to $x$, where we can apply the
locally minimizing property for $c(t)$, which also proves the minimizing property of $c(t)$ from $x_1$ to $x_2$.
\end{proof}

Using the variational method we can get  the edge-crossing equation for geodesics.

\begin{lemma}\label{lemma-geodesic-crossing-edge}
 Let $\{c(t),t\in(-\infty,\infty)\}$ be a geodesic in $T_xM$ crossing the line $E$ from $(C_{x,1},F_1)$ to $(C_{x,2},F_2)$ at $c(0)=x$. Then
we have
\begin{equation}\label{geodesic-equation-crossing-edge}
\langle \dot{c}^-(0),v\rangle_{\dot{c}^-(0)}^{F_1}
=\langle \dot{c}^+(0),v\rangle_{\dot{c}^+(0)}^{F_2}.
\end{equation}
\end{lemma}

\begin{proof} By Lemma \ref{lemma-global-minimizing}, the geodesic $c(t)$
is minimizing from $x'=c(-1)$ to $x''=c(1)$. Consider a continuous family $\{p_s(t),t\in [-1,1]\}$ of paths from $x'=p_s(-1)$ to $x''=p_s(1)$ defined for all $s$ in a small neighborhood of $0$, such that each path consists of two straight line segments
$L_{x',x_s}$ and $L_{x_s,x''}$, where $x_s\in E$ and the vector from $x$ to $x_s$ is $v_{x,x_s}=sv$. The arc length of $\{p_s(1),t\in[-1,1]\}$ is then
$f(s)=F_1(\dot{c}^-(0)+sv)+F_2(\dot{c}^+(0)-sv)$. It is minimizing when $s=0$,
so $$f'(0)=\langle \dot{c}^-(0),v\rangle_{\dot{c}^-(0)}^{F_1}
-\langle \dot{c}^+(0),v\rangle_{\dot{c}^+(0)}^{F_2}=0,$$
which completes the proof of the lemma.
\end{proof}

We stress that
(\ref{geodesic-equation-crossing-edge}) is crucial for the study of piecewise flat Finsler surfaces. In the following, we  will refer to  it as {\it the  edge-crossing equation}.
We have seen that it must be satisfied for any geodesic crossing an edge. By Lemma \ref{geodesic-extension-crossing-edge} below, we will see, it can be use to equivalently define the gedesics locally around
an edge point which is not a vertex. Globally, any geodesic
on a piecewise flat surface, which does not pass any vertex
is a union of line segment $L_{x_i,x_{i+1}}$, such that
each $x_i$ is either an endpoint of the geodesic, or a point
on some edge where the edge-crossing equation is satisfied.

In our further study of (\ref{geodesic-equation-crossing-edge}),  the following convexity lemma
for Minkowski norms will be useful.

\begin{lemma} \label{lemma-convexity}
Let $(\mathbf{V},F)$ be a real vector space endowed with
a Minkowski norm. Assume $u$, $v$ and $w$ are the $F$-unit vectors such that $u$ is a non-negative linear combination of $v$ and $w$. Then we have $\langle u,w\rangle_u^F\geq\langle v,w\rangle_v^F$, where the
equality holds if and only if  $u=v$.
\end{lemma}

\begin{proof} Obviously we only need to prove the lemma when $\dim\mathbf{V}=2$.
Fix the $F$-unit vector $w$ and denote $w'$ the $F$-unit vector which is opposite to $w$. The Legendre transformation which maps $F$-unit
$u\in\mathbf{V}$ to the $F^*$-unit vector $\mathcal{L}(u)=\langle u,\cdot\rangle_u^F\in\mathbf{V}^*$ is a diffeomorphism from the indicatrix
$\mathcal{S}^{F}\subset\mathbf{V}\backslash\{0\}$ to the indicatrix $\mathcal{S}^{F^*}\subset\mathbf{V}^*\backslash\{0\}$.
Evaluation at $w$ can be viewed as a linear function on
$\mathbf{V}^*$, whose kernel is parallel to the tangent space of $\mathcal{S}^{F^*}$ at $\mathcal{L}(w)$ and $\mathcal{L}(w')$.
So for any monotonous curve $w^*(t)\in\mathcal{S}^{F^*}$ from
$\mathcal{L}(w')$ to $\mathcal{L}(w)$, $w^*(t)(w)$ is strictly increasing.

When two of $u$, $v$, $w$ and $w'$ are equal, the proof for the
lemma is easy. So we may assume that they are distinct. By the assumption in
the lemma,  $u$ is a positive
linear combination of $v$ and $w$, so $u$ and $v$ are contained in a monotonous curve $w(t)$ on $\mathcal{S}^F$,
such that $w(0)=w'$, $w(t_1)=v$, $w(t_2)=u$, and $w(t_3)=w$ with $0<t_1<t_2<t_3$. Correspondingly $\mathcal{L}(w(t))$ is a monotonous curve
on $\mathcal{S}^{F^*}$ from $\mathcal{L}(w')$ to $\mathcal{L}(w)$.
Our previous observation indicates
$$\langle u,w\rangle_u^F=\mathcal{L}(w(t_1))(w)<\mathcal{L}(w(t_2))(w)=
\langle v,w\rangle_v^F,$$
which proves the lemma.
\end{proof}

Lemma \ref{lemma-convexity} implies the following monotonous relation between $\dot{c}^-(0)$ and $\dot{c}^+(0)$.

\begin{lemma}\label{mono-lemma-0}
Let $\{c_1(t),t\in(-\infty,\infty)\}$ and $\{c_2(t),t\in(-\infty,\infty)\}$
be two different geodesics in $T_xM$, crossing $E=C_{x,1}\cap C_{x,2}$ from $C_{x,1}$ to $C_{x,2}$, with $c_1(0)=c_2(0)=x$. Let $v$ be any nonzero vector on $E$. Then
$\angle_x^-(\dot{c}_1^-(0),v)>\angle_x^-(\dot{c}_2^-(0),v)$ if and only if
$\angle_x^+(\dot{c}_1^+(0),v)>\angle_x^+(\dot{c}_2^+(0),v)$.
\end{lemma}

\begin{proof}
Assume $\angle_x^-(\dot{c}_1^-(0),v)>\angle_x^-(\dot{c}_2^-(0),v)$.
Then $\dot{c}_2^-(0)$ is a positive linear combination of $\dot{c}_1^-(0)$
and $v$. By Lemma \ref{lemma-convexity} and Lemma \ref{lemma-geodesic-crossing-edge}, we have
$$\langle \dot{c}_2^+(0),v\rangle_{\dot{c}_2^+(0)}^{F_2}
=\langle \dot{c}_2^-(0),v\rangle_{\dot{c}_2^-(0)}^{F_1}
<\langle \dot{c}_1^-(0),v\rangle_{\dot{c}_1^-(0)}^{F_1}
=\langle \dot{c}_1^+(0),v\rangle_{\dot{c}_1^+(0)}^{F_2},$$
so by Lemma \ref{lemma-convexity} again, in $\mathcal{S}_{x,2}^+ M$,
$-v/F(-v)$, $\dot{c}_1(0)^+$, $\dot{c}_2(0)^+$ and $v/F(v)$ are arranged
with the order presented above,  i.e.,  we also have
$\angle_x^+(\dot{c}_1(0)^+,v)>\angle_x^+(\dot{c}_2(0)^+,v)$.

The proof for the other direction is similar.
\end{proof}

The edge-crossing equation guarantees that any geodesic with an end point $x\in M$ on an edge but not a vertex can be further
extended.

\begin{lemma} \label{geodesic-extension-crossing-edge}
Let $\{c(t),t\in(-\infty,0]\}\subset C_{x,1}$
(or $\{c(t),t\in[0,\infty)\}\subset C_{x,2}$) be a unit speed geodesic in $T_xM$ with $c(0)=x$. Then it can be uniquely extended to a geodesic for $t\in(-\infty,\infty)$,
such that $\dot{c}^\pm(0)$ are related by (\ref{geodesic-equation-crossing-edge}).
\end{lemma}
\begin{proof} We first extend the geodesic $\{c(t),t\in(-\infty,0]\}$ to a unit speed curve for $t\in(-\infty,\infty)$, which is the ray $R_{x,\dot{c}^+(0)}$
when $t>0$, satisfying
(\ref{geodesic-equation-crossing-edge}).

First we need to prove the existence of such an extension, i.e. the existence of $\dot{c}^+(0)$ satisfying (\ref{geodesic-equation-crossing-edge}).
Obviously when $\dot{c}^-(0)$ is tangent to $E$, then $\dot{c}^+(0)=\dot{c}^-(0)$. Implied by Lemma \ref{mono-lemma-0} and the implicit function theorem, the correspondence between $\dot{c}^-(0)$ to $\dot{c}^+(0)$ is a
homeomorphism between $\mathcal{S}_{x,1}^-M$ and $\mathcal{S}_{x,2}^+M$. So for any $\dot{c}^-(0)$, we can find
a unique $\dot{c}^+(0)$, such that the edge-crossing equation (\ref{geodesic-equation-crossing-edge}) is satisfied.

Next we prove such an extension $\{c(t),t\in(-\infty,\infty)\}$ is a geodesic. We only need to prove that it is minimizing from $x'=c(-1)$ to $x''=c(1)$.

Let $v$ be any nonzero vector on $E=C_{x,1}\cap C_{x,2}$.
We claim that any piecewise smooth path from $x'$ to $x''$
passing $x_s$ on $E$ such that $v_{x,x_s}=sv$ has an arc length no
less than $$f(s)=F_1(\dot{c}(0)^-+sv)+F_2(\dot{c}(0)^+-sv).$$
Assume conversely that it is not true. Then we can perturb the path to the union of a sequence of line segments $L_{x',x_1}\cup\mathop{\bigcup}\limits_{i=1}^{n-1} L_{x_i,x_{i+1}}\cup
L_{x_n,x''}$ such that $x_s$ equals some $x_i$, and the total length is still less than $f(s)$. Using the triangle inequality
repeatedly, we get $l(L_{x',x_s})+l(L_{x_s,x''})<f(s)=l(L_{x',x_s})+l(L_{x_s,x''})$, which is a contradiction.

Now we prove $s=0$ is the only minimum point for the smooth function $f(s)$. By Lemma \ref{lemma-geodesic-crossing-edge}, it is a critical point for $f(s)$. On the other hand, easy calculation shows
$f(s)$ is strict convex function, i.e. whenever $s_1\neq s_2$,
\begin{eqnarray*}& &
f(s_1)+f(s_2)\\
&=&(F_1(\dot{c}^-(0)+s_1v)+F_1(\dot{c}^-(0)+s_2v))
+(F_2(\dot{c}^+(0)-s_1v)+F_2(\dot{c}^+(0)-s_2v))\\
&>&2(F_1(\dot{c}^-(0)+\frac{s_1+s_2}{2}v)+
F_2(\dot{c}^+(0)-\frac{s_1+s_2}{2}v))\\
&=&2f(\frac{s_1+s_2}{2}).
\end{eqnarray*}
So $s=0$ is the only critical point as well as the only minimum point of $f(s)$. This proves the curve $R_{x,\dot{c}(0)^-}\cup R_{x,\dot{c}(0)^+}$ is a geodesic from $x'$ to $x''$.

The proof for extending $\{c(t),t\in[0,\infty)\}$ is similar.
\end{proof}

Summarizing the above observations, we have
\begin{theorem}\label{theorem-000}
For any two different points $x'$ and $x''$ in the tangent cone $T_xM$,
where $x$ is an edge point but not a vertex, there exists a unique unit speed geodesic from $x'$ to $x''$.
\end{theorem}

\begin{proof}
Assume $T_xM=C_{x,1}\cup C_{x,2}$ and the same straight line $E=C_{x,1}\cap C_{x,2}$. Let $v$ be any nonzero vector in the direction of $E$. Let $v$ be any nonzero vector on $E$.
If $x'$ and $x''$ are contained in the same $C_{x,i}$, then the only geodesic
from $x'$ to $x''$ must be the straight line segment $L_{x',x''}$.

Suppose $x'$ and $x''$ are not  contained in the same $C_{x,i}$. Without loss of generality, we assume $x'\in C_{x,1}$ and $x''\in C_{x,2}$. Consider a family of paths from $x'$ to $x''$, given by $L_{x',x_s}\cup
L_{x_s,x''}$, where $x_s$ is a point on $E$ with $v_{x,x_s}=sv$. Its arc length is
$f(s)=F_1(v_{x',x}+sv)+F_2(v_{x,x''}-sv)$. Since
$\mathop{\lim}\limits_{s\rightarrow\pm\infty}f(s)=+\infty$,  it must reach a minimum, at $s=s_0$ for example. Then the argument in Lemma \ref{lemma-geodesic-crossing-edge}
shows the edge-crossing equation (\ref{geodesic-equation-crossing-edge}) is satisfied at $s=s_0$.
Meanwhile, the argument in Lemma \ref{geodesic-extension-crossing-edge} shows $L_{x',x_{s_0}}\cup L_{x_{s_0},x''}$
is the only minimizing geodesic from $x'$ to $x''$.
\end{proof}

The method in the proof of Lemma \ref{geodesic-extension-crossing-edge}  can be applied to discuss the property of a geodesic passing a finite sequence of Minkowski triangles $(T_i,F_i)$ and edges $E_i$.
Let $c(t)$ be a geodesic
$\mathop{\bigcup}\limits_{i=0}^{n-1} L_{x_i,x_{i+1}}$
from $x_0\in (T_0,F_0)$ to $x_{n}\in (T_{n-1},F_{n-1})$, passing the edges $E_i=T_{i-1}\cap T_{i}$, $i=1,\ldots,n-1$, at $x_i\in E_i$ which is not a vertex.
Denote  $u_i=v_{x_{i},x_{i+1}}$ the vector in the Minkowski triangle $T_i$. Let $v_i$ be a nonzero
vector on each $E_i$. Then by Lemma \ref{lemma-geodesic-crossing-edge}, we have
\begin{equation}\label{0000}
\langle \frac{u_i}{F_i(u_i)},v_i\rangle^{F_i}_{u_i}=
\langle \frac{u_{i+1}}{F_{i+1}(u_{i+1})},v_i\rangle^{F_{i+1}}_{u_{i+1}},
\quad \forall i=0,\ldots,n-1.
\end{equation}
For a real parameter $s$, we
consider another sequence of points $\{x'_i\}_{0\leq i\leq n}$, such that $x'_0=x_0$, $x'_{n}=x_{n}$,   $x'_i\in E_i$ for $1\leq i\leq n-1$, and
$v_{x_i,x'_i}=sv_i$ (this requirement determine the interval $I$ of all possible $s$). Connecting all the line segments $L_{x'_i,x'_{i+1}}$, we get another path from $x_0$ to $x_n$, whose
arc length is
$$f(s)=F_0(u_0+sv_1)+\sum_{i=1}^{n-2} F_i(u_i+s(v_{i+1}-v_{i}))+F_{n-1}(u_n-sv_{n-1}).$$
Similar argument as in the proof of Lemma
\ref{geodesic-extension-crossing-edge} shows $f(s)$ is
a strict convex function and $s=0$ is the only critical point of $f$, so we have
\begin{lemma}\label{lemma-2}
Let $\{p_s(t),x_0,x_n\}$ be the path constructed above from $x_0$ to $x_n$. Assume there exists an interval $I$ containing $0$, such that the union of the paths $\{p_s(t),x_0,x_n\}$ defined above
 for $s\in I$ contains no vertices. Then the arc length function $f(s)$ is strictly convex and has the unique critical point at $s=0$, i.e. it strictly increases for $s\geq 0$, and strictly decreases for $s\leq 0$.
\end{lemma}

\subsection{The Landsberg  and Berwald conditions}

The edge-crossing equation (\ref{geodesic-equation-crossing-edge})
in Lemma \ref{lemma-geodesic-crossing-edge}
defines a continuous relation between $\dot{c}^-(0)$ and $\dot{c}^+(0)$.
By Lemma \ref{mono-lemma-0} it defines a homeomorphism between $\mathcal{S}_{x,1}^-M$ and
$\mathcal{S}_{x,2}^+M$. By the implicit function theorem, it is in fact
a diffeomorphism when $\dot{c}^\pm(0)$ are not parallel to $E$.

The edge-crossing equation can also be used to define some special types
of piecewise flat Finsler surfaces. In particular, we can define the notions of \textit{Landsberg} and \textit{Berwald} surfaces in this setting.
\begin{definition}\label{define-landsberg-berwald}
 A piecewise flat Finsler surface $M$ is   called  a piecewise flat Landsberg surface  if for any $x$ on an edge $E$ between two Minkowski triangles $(T_1,F_1)$ and
$(T_2,F_2)$ in $M$, which is not a vertex, the homeomorphism between
$u_1^-\in\mathcal{S}_{x,1}^-M$ and $u_2^+\in\mathcal{S}_{x,2}^+M$ defined by
$$\langle u_1^-,v\rangle^{F_1}_{u_1^-}=\langle u_2^+,v\rangle^{F_2}_{u_2^+},$$
where $v$ is any nonzero tangent vector in the direction of $E$,
is an isometry with respect to the metrics on the indicatrices defined by the Hessian matrices. It is called  a piecewise flat Berwald surface if the homeomorphism between $u_1^-$ and $u_2^+$
is induced by a linear isomorphism between the tangent spaces $T_x(T_1)$ and
$T_x(T_2)$.
\end{definition}

Notice that the positions of $(T_1,F_1)$ and $(T_2,F_2)$ can be switched, so at any edge point $x$ described in Definition \ref{define-landsberg-berwald}, the Landsberg condition gives an isometry between
$\mathcal{S}_x(T_1)=\mathcal{S}_{x,1}^-M\cup\mathcal{S}_{x,1}^+M$
and $\mathcal{S}_x(T_2)=\mathcal{S}_{x,2}^+M\cup\mathcal{S}_{x,2}^-M$ ( hence also between
$T_x(T_1)\backslash \{0\}$ and
$T_x(T_2)\backslash \{0\}$ ) with respect to the metrics defined by the Hessian matrices. So the Hessian matrices define the same metrics at all smooth points (i.e. the complement of all edges and vertices) of $M$ when $M$ is connected. This is a key property of smooth Landsberg Finsler spaces. When $M$ is Berwald, the non-vertex edge points are smooth points as well. Since the isometry between $T_x(T_1)\backslash \{0\}$ and
$T_x(T_2)\backslash \{0\}$ is  linear,  $(T_1,F_1)$ and $(T_2,F_2)$ can be isometrically embedded in the same Minkowski plane, as
two triangles with a common edge. So the complement of all vertices in $M$ is
a smooth incomplete Berwald surface.

Examples of piecewise flat Berwald surfaces include piecewise flat Riemannian surfaces. However, there are also many Non-Riemannian examples of Berwald type. For example, given a  piecewise flat Riemannian surface, such that all the angles in
the triangulation is an integer multiple of $\pi/n$ for some $n\in\mathbb{N}$, we can replace the Euclidean norms on the triangles  with
 non-Riemannian $D_{2n}$-invariant Minkowski norms (here $D_{2n}$ is the dihedral group),  which makes $M$
 a non-Riemannian piecewise flat Finsler surface of Berwald type.
Till now, we do not know whether any Landsberg surface $M$ is Berwald, i.e., whether the answer to the combinational version of the $2$-dimensional Landsberg problem (the unicorn problem) is positive.

\section{Geodesics in the tangent cone $T_xM$ for a vertex $x$}

In this section, we study the local behavior of a geodesic near a vertex $x$,
so the discussion can be restricted to the tangent cone $T_xM$.
We will use the following lemma.

\begin{lemma} \label{lemma-reversible-convexity}
Let $(\mathbf{V},F)$ be a real vector space endowed with
a Minkowski norm, and  $u$, $v$ and $w$ be $F$-unit vectors such that $u$ is a non-negative linear combination of $v$ and $w$ and one of the following two conditions:
\begin{description}
\item{\rm (1)} $F$ is reversible, i.e. $F(y)=F(-y)$ for all $y\in\mathbf{V}$.
\item{\rm (2)}
$\langle v,w\rangle_w^F\geq 0$.
\end{description}
Then we have $\langle u,w\rangle^F_w\geq\langle v,w\rangle_w^F$, where the
equality holds if and only if $u=v$.
\end{lemma}

\begin{proof}
We first assume $F$ is reversible.
Similarly as in the proof of Lemma \ref{lemma-convexity}, we can assume that $\dim\mathbf{V}=2$. In the cases that $u=\pm w$, $v=\pm w$ or $u=v$, the lemma is obvious.
So we may assume that $u$, $v$, $w$ and $-w$ are distinct (in this case $u$
is a positive linear combination of $v$ and $w$). Then we only need to prove
$\langle u-v,w\rangle^F_w>0$. The
indicatrix $\mathcal{S}$ in $\mathbf{V}$ is a smooth circle which bounds a
strictly convex domain.
The vectors $u$ and $v$ are contained in the same connected component of $\mathcal{S}\backslash\{\pm w\}$, which can be parametrized by its arc length
as $w(t)$ with $w(0)=-w$, $w(t_1)=v$, $w(t_2)=u$, and $w(t_3)=w$. Since  $u$
is a positive linear combination of $v$ and $w$, we have $0<t_1<t_2<t_3$. Since $F$ is reversible, the tangle lines of $\mathcal{S}$ at $\pm w$ are parallel. So
from the strong convexity of $\mathcal{S}$, we have $\langle w'(t),w\rangle^F_w>0$ for all $t\in (0,t_3)$, that is,  $f(t)=\langle w(t),w\rangle^F_w$ is strictly increasing when $t\in [0,t_3]$,
completing  the proof of the lemma.

This proof is still valid when the reversibility condition is replaced by $\langle v,w\rangle_w^F\geq 0$,
because $\langle w'(t),w\rangle_w^F>0$ for
$t\in (t_1,t_3)$ though it may fail for $t\in (0,t_1)$.
\end{proof}

The main theorem of this section is the following.

\begin{theorem}\label{lemma-geodesic-not-passing-vertex-x}
Let $M$ be a piecewise flat Finsler surface and $T_xM$ the tangent cone of a vertex $x\in M$.
 Let  $c(t)$ be a geodesic in $T_xM$. Suppose that there exists $x'=c(t')\neq x$, such that the tangent vector $\dot{c}^+(t')$ or $\dot{c}^-(t')$
at $x'$ is not in the ray at $x$. Then $c(t)$ can be uniquely extended to a geodesic defined on $(-\infty,\infty)$,  and it  is a union of finite line segments and two rays as the following: \begin{equation}\label{decomposition-of-geodesic}
\{c(t),t\in(-\infty,\infty)\}=
R_{x_1,v_1^-}\cup(\bigcup_{i=1}^{m-1} L_{x_i,x_{i+1}})\cup R_{x_m,v_m^+},
\end{equation}
where $R_{x_1,v_1^-}$, $L_{x_1,x_2}$, $\ldots$, $L_{x_{m-1},x_m}$ and
$R_{x_m,v_m^+}$ are contained in the Minkowski cones $C_{x,0}$, $\ldots$, $C_{x,m}$ respectively, and each $x_i\neq x$ is contained in the edge
$E_{i}=C_{x,i-1}\cap C_{x,i}$. In particular,  the geodesic
$\{c(t),t\in(-\infty,\infty)\}$ does not pass $x$, and neither of the line segments and rays in (\ref{decomposition-of-geodesic}) is contained in an edge.
\end{theorem}

\begin{proof}
We first extend the geodesic $c(t)$ in the positive direction.
Within the same Minkowski cone as $x'=c(t')$, the geodesic $c(t)$ intersects
the line segment $L_{x,c(t)}$ transversally at each $c(t)$. The intersection of $c(t)$
with the edge is not $x$. After crossing the edge, by Lemma \ref{geodesic-extension-crossing-edge}, $c(t)$ still intersects the line segment
$L_{x,c(t)}$ transversally. So whenever the geodesic $c(t)$ can be extended, it does not pass $x$.

The proof of the theorem is reduced to the following key observation:

\medskip
{\bf Assertion (A).}\quad  The geodesic $\{c(t),t\in[t',+\infty)\}$ only intersects the edges in $T_xM$ for finitely many times when $t$ increases.

\medskip
Now we switch to prove the Assertion (A).

Suppose  there are exactly $n$ different Minkowski cones and $n$ different edges in $T_xM$. We denote them as $(C_{x,i},F_i)$ and $E_i=C_{x,i-1}\cap C_{x,i}$ for $i\in\mathbb{Z}$, and $(C_{x,i},F_i)=(C_{x,i+n},F_{i+n})$ and $E_i=E_{i+n}$ for all $i$. Let $v_{i}$ be the unit vector on the edge $E_{i}$ in the direction from $x$ to $x_i$ (we also have $v_i=v_{i+n}$ for all $i$).

Without loss of generality,
we may assume that
$x'=c(t')\in C_{x,0}\backslash (E_0\cup E_1)$ and $c(t)$ intersects $E_1$ when $t>t'$. Then with $t$ increasing,
the edges $c(t)$ intersects are sequentially $E_1$, $E_2$, $E_3$ and so on. Denote $\theta'=\angle_{x'}^+(v_1,\dot{c}(t'))$. The range of $\theta'$ is
$$0\leq\theta'<\alpha=
\angle_{x'}^+(v_1,v_{x',x}),$$ where the supremum corresponds
to the line segment $L_{x',x}$.

When $\theta$ is small, the extension of the geodesic $c(t)$ can not intersect $E_1$ again. Assersion (A) is obviously true for this case.

When $\theta$ is big enough such that $c(t)$ intersect $E_1$ again, we can find $x''=c(t'')\in C_{x,0}\backslash(E_0\cup E_1)$  from the first return of $c(t)$ back to $C_{x,0}$. Then we define
$\theta''=\angle_{x''}^+(v_1,\dot{c}(t''))$, and claim

\medskip
{\bf Assertion (B).}\quad  $\theta''<\theta'$.

\medskip

Assuming Assertion (B), we can prove Assertion (A) as following. Let $\mathcal{A}$ be the subset of $[0,\alpha)$
defined by all the number $a$ such that whenever the geodesic $\{c(t),t\in[t',+\infty)\}$ with
$c(t')=x'\in C_{x,0}\backslash(E_0\cup E_1)$ satisfies $\theta'<a$, it intersects only finite edges. It is easy to see $\mathcal{A}$ is an interval containing 0. Using the edge-crossing equation repeatedly, it is not hard to observe that, for any geodesic $\{c(t),t\in[t',+\infty)\}$ with only finite intersections with the edges, we can change it by increasing its $\theta'$ a little bit, such that the number counting the intersection with edges remain the same or increase by 1, i.e. still finite. So $\mathcal{A}=[0,\beta)$ for some $\beta\in(0,\alpha]$. If $\beta<\alpha$, by
Assertion (B), $\beta$ is also contained in $\mathcal{A}$, which is a contradiction. So $\mathcal{A}=[0,\alpha)$, and
Assertion (A) is proved.

Now we only need to prove Assertion (B).

By the edge-crossing equation (\ref{geodesic-equation-crossing-edge}), $\theta''$ only depends on $\theta'$ rather than $x'=c(t')$. For $\theta''$ to exist, $\theta'$ must be taken from an open interval $(\gamma,\alpha)$,
where $\gamma>0$ corresponds to a geodesic $c(t)$ with $t\geq t'$
which ends with a ray parallel to $E_1$ in $C_{x,0}$. When $\theta'\in(\beta,\alpha)$ is sufficiently close to $\gamma$,
$\theta''=\angle_{x''}^+(v_1,\dot{c}(t''))$ can be arbitrarily close to 0.

So if we assume conversely that Assertion (B) is not true, then
by the intermediate value theorem, there exist a geodesic $c(t)$ satisfying $\theta'=\theta''$, and then $\dot{c}(t')=\dot{c}(t'')\theta'$. We denote $\{x_i\in E_i, 1\leq i\leq n+1\}$ the first $n+1$
intersection points of $\{c(t),t\in[t',+\infty)\}$ with the edges, and $u_i=v_{x_i,x_{i+1}}/F_{i}(v_{x_i,x_{i+1}})$ for $0<i\leq n$ the $F_{i}$-unit vector in the direction from $x_i$ to $x_{i+1}$. The assumption $\theta''=\theta'$ implies
$u_0=v_{x',x_1}/F_0(v_{x',x_1})=u_n$ and $u_{i+n}=u_i$ can also be defined.

By the periodicity, there exists an $i$ with the biggest
$\langle u_i,v_i\rangle_{u_i}^{F_i}$. If $\langle u_i,v_i\rangle_{u_i}^{F_i}\geq 0$, then
$$\langle u_i,v_{i+1}\rangle_{u_i}^{F_i}=\langle u_{i+1},v_{i+1}\rangle_{u_{i+1}}^{F_{i+1}}\leq\langle u_i,v_i\rangle_{u_i}^{F_i},$$
which is a contradiction with Lemma \ref{lemma-reversible-convexity} because $v_{i+1}$ is
a positive linear combination of $v_i$ and $u_i$.
So we have $\langle u_i,v_i\rangle_{u_i}^{F_i}<0$ for each $i$.

Denote $v'_i$ the unit vector in the opposite direction of $v_i$, i.e. $v_i=c_iv'_i$ for some $c_i<0$ and
$F_i(v'_i)=F_{i-1}(v'_i)=1$. Then for each $i$, $v'_i$ is a positive linear combination of $v'_{i+1}$ and $u_i$, and
$\langle u_i,v'_i\rangle_{u_i}^{F_i}>0$ for each $i$. We can choose $i$ with the smallest $\langle u_i,v'_i\rangle_{u_i}^{F_i}$, then
$$\langle u_i,v'_{i+1}\rangle_{u_i}^{F_i}=\langle u_{i+1},v'_{i+1}\rangle_{u_{i+1}}^{F_{i+1}}\geq\langle u_i,v'_i\rangle_{u_i}^{F_i}.$$
This is a contradiction with Lemma \ref{lemma-reversible-convexity}, which ends the proof of Assertion (B).

To summarize, the extension for the geodesic in the positive direction can only intersect the edges for finitely many times, it will provide a geodesic
for $t\in[t_0,\infty)$, which contains finite many line segments and a ray.

For the extension of the geodesic in the negative direction, the argument is similar. To summarize, in $T_xM$, $c(t)$ can
be extended to a geodesic for $t\in(-\infty,\infty)$,
not passing $x$, and can be presented as the union of two rays
and finitely many Line segments.
\end{proof}

Using Theorem \ref{lemma-geodesic-not-passing-vertex-x},
we can provide an explicit description of all maximally extended geodesics in
$T_xM$ of a vertex $x$.

\begin{corollary} \label{cor-1}
Let $M$ be a piecewise flat Finsler surface and $T_xM$ be the tangent cone of a vertex $x\in M$. Then
a maximally extended geodesic in $T_xM$ must be one of the following:
\begin{description}
\item{\rm (1)} A geodesic described in Theorem \ref{lemma-geodesic-not-passing-vertex-x}.
\item{\rm (2)} A ray $R_{x,v^\pm}$, where $v^\pm$ is a nonzero incoming or outgoing tangent vector at $x$.
\item{\rm (3)} The union of two rays $R_{x,v^-}\cup R_{x,v^+}$, where $v^-$
and $v^+$ are nonzero incoming and outgoing tangent vectors at $x$ respectively.
\end{description}
\end{corollary}

The geodesics in (1) can be easily distinguished from those in (2) and (3). The crucial question is how to distinguish the geodesics in the cases (2) and (3).

\begin{question}\label{question}
Let $M$ be a piecewise flat Finsler surface and $T_xM$  the tangent cone of a vertex $x\in M$.
Let $\{c(t),t\in(-\infty,0]\}$ or $\{c(t),t\in[0,+\infty)\}$ be a unit speed geodesic in $T_xM$ with $c(0)=x$. Under what condition can one extend it to a geodesic defined on the whole $(-\infty,\infty)$?
\end{question}

Before answering this question in the next section, some preparation is needed here. We now consider several different types perturbations of a geodesic passing the vertex $x$, which are important for answering Question \ref{question} as well as defining curvature.

We first define the perturbations of $\{c(t),t\in(-\infty,0]\}\subset T_xM$ with $c(0)=x$, which are geodesics in (1) of Corollary \ref{cor-1}.
For $x_0=c(-1)$, we denote $\mathcal{S}_{x_0}^{+,l}$  and $\mathcal{S}_{x_0}^{+,r}$ the two connected components of $\mathcal{S}_{x_0}^+(T_xM)
\backslash\mathbb{R}\dot{c}(-1)$.  The superscript $l$ or $r$ means \textit{left side} or \textit{right side}, respectively. Usually we can take an orientation on $T_xM$
(i.e. a local orientation at $x\in M$) to specify the \textit{left} and \textit{right} sides of the geodesic $c(t)$
at $x_0$. If not specified, the choice of the local orientations will have no impact on later discussions or definitions.

Given a  unit outgoing tangent vector $w^+\in\mathcal{S}_{x_0}^{+,l}$
(or $w^+\in\mathcal{S}_{x_0}^{+,r}$) at $x_0$, we can define a unit speed geodesic
$\{c_{x_0,w^+}(t),t\in[-1,\infty)\}$ with $c_{x_0,w^+}(-1)=x_0$ and
$\dot{c}_{x_0,w^+}^+(-1)=w^+$. Notice that such a  geodesic can be extended to one with $t\in(-\infty,\infty)$, which is a geodesic not passing $x$, as described in Lemma \ref{lemma-geodesic-not-passing-vertex-x}, but here we only consider the part for $t\in[-1,\infty)$.

Another way to perturb the geodesic $c(t)$ is given by a sufficiently small  parallel shifting of the geodesic
$\{c(t),t\in(-1-\varepsilon,-1+\varepsilon)\}$ with sufficiently small $\varepsilon>0$, in the direction of $w^+$, where $w^+\notin\mathbb{R}\dot{c}(0)$ is a unit outgoing tangent vector at $x_0$. Then the new geodesic can be extended to $(-\infty,\infty)$, which is a geodesic not passing $x$, as described in
Lemma \ref{lemma-geodesic-not-passing-vertex-x}. We denote it as
$\{c_{+,l}(t),t\in(-\infty,\infty)\}$ or $\{c_{+,r}(t),t\in(-\infty,\infty)\}$,
 according to $w^+\in\mathcal{S}_{x_0}^{+,l}$ or $w^+\in\mathcal{S}_{x_0}^{+,r}$ respectively. Here we add the subscript $+$ since they are  the extension of the geodesic in the positive direction. In this notation for the perturbation geodesic, we do not need to specify the exact choice of $w^+$, because only the sides for it are relevant to later discussion.

The following lemma is important for discussing curvature in the next section.
\begin{lemma}\label{lemma-monotonous}
Let $M$ be a piecewise flat Finsler surface and $T_xM$  the tangent cone of a vertex $x\in M$. Keep all the  notations as above. Let  $\{c(t),t\in(-\infty,0]\}$ be the geodesic
 with $c(0)=x$ and $c(-1)=x_0$. Then we have the following:
\begin{description}
\item{\rm (1)} If  $w^+, w'^+\in\mathcal{S}_{x_0}^{+,l}$ (or $w^+, w'^+\in\mathcal{S}_{x_0}^{+,r}$) satisfy the condition $$\angle_{x_0}^+(w^+,\dot{c}(-1))> \angle_{x_0}^+(w'^+,\dot{c}(-1)^+),$$ then we have
    $$\sphericalangle_{x}^\pm
    (\{c_{x_0,w^+}(t),t\in[-1,\infty)\})<
    \sphericalangle_{x}^\pm
    (\{c_{x_0,w'^+}(t),t\in[-1,\infty)\}).$$  In particular,
    any ray from $x$ intersects  $\{c_{x_0,w'^+}(t),t\in[-1,\infty)\}$ if
    it intersects  $\{c_{x_0,w^+}(t),t\in[-1,\infty)\}$.
\item{\rm (2)} The angles $\sphericalangle_{x}^\pm
    (\{c_{+,l}(t),t\in(-\infty,\infty)\})$ are the limits of $\sphericalangle_{x}^\pm
    (\{c_{x_0,w^+}(t),t\in[-1,\infty)\})$  when $w^+\in\mathcal{S}_{x_0}^{+,l}$ approaches  $\dot{c}(-1)$. In particular, $\sphericalangle_{x}^\pm
    (\{c_{+,l}(t),t\in(-\infty,\infty)\})$ only depends on the side of the parallel shifting. A ray from $x$ intersects  $\{c_{+,l}(t),t\in(-\infty,\infty)\}$ iff it intersects  $\{c_{x_0,w^+}(t),t\in(-1,\infty)\}$ for some $w^+\in\mathcal{S}_{x_0}^{+,l}$. The similar assertion is also true for perturbations to the right side.
\end{description}
\end{lemma}

\begin{proof}
(1) According to Lemma \ref{lemma-geodesic-not-passing-vertex-x}, we can assume that $\{c_{x_0,w^+},t\in[-1,\infty)\}$ is the union
$$L_{x_0,x_1}\cup(\bigcup_{i=1}^{m-1}L_{x_i,x_{i+1}})\cup
R_{x_m,u_m},$$
where
$x_0\in (C_{x,0},F_0)$, and the line segments $L_{x_i,x_{i+1}}$ and the ray  $R_{x_m,u_m}$ satisfy  the following conditions: for $0\leq i\leq m-1$, $L_{x_i,x_{i+1}}$ is contained in $(C_{x,i},F_i)$ but not in an edge, for $1\leq i\leq m$, $x_i=c_{x_0,w^+}(t_i)\neq x$ is on the edge $E_i=C_{x,i-1}\cap C_{x,i}$,
and for $0\leq i\leq m$,
$u_i$ is the unit outgoing tangent vector at $x_i$ for $L_{x_i,x_{i+1}}$ and $R_{x_m,u_m}$, i.e. $u_i=v_{x_i,x_{i+1}}/F_{i}(v_{x_i,x_{i+1}})$ when $i<m$. Denote $v_i$ the unit tangent vector in the direction from $x$ to $x_i$,  $1\leq i\leq m$,
and in particular $v_0=-\dot{c}(-1)/F_0(-\dot{c}(-1))$.

The geodesic $\{c_{x_0,w'^+}(t),t\in[-1,\infty)\}$ can be similarly presented as the union of $m'$ straight line segments $L_{x'_i,x'_{i+1}}$ and the ray $R_{x'_{m'},u'_{m'} }$. Notice that it is a perturbation of $c(t)$ at the same side with $c_{x_0,w^+}(t)$, so we have
$x'_0=x_0$, and moreover, for $1\leq i\leq \min\{m,m'\}$, $x'_i=c_{x_0,w'^+}(t'_i)\neq x$ is on the same edge $E_i$ as $x_i$, and when $0\leq i\leq \min\{m,m'\}$, $L_{x'_{i},x'_{i+1}}$ (or $R_{x'_{m'},u'_{m'} }$ when $i=m'$) is contained in the same $C_{x,i}$ as $L_{x_{i},x_{i+1}}$ (or $R_{x_m,u_m^+}$ when $i=m$). Correspondingly, for $0\leq i\leq m'-1$, we have unit outgoing vectors
$u'_i=v_{x'_i,x'_{i+1}}/F_i(v_{x'_i,x'_{i+1}})$
at $x'_i$ from $x'_i$ to $x'_{i+1}$ for each $i<m'$.

First we prove $m\leq m'$.
Our assumption on the two geodesics implies that
$u_0=w^+$ is a positive linear combination of $u'_0 =w'^+$ and $v_0=-\dot{c}(-1)$. If $m=0$, it is done. Otherwise there exists an intersectional point $x_1$.
Then $v_1$ is a positive linear combination of $v_0$ and $u_0$. Combining these two observations, we conclude that
$v_1$ is a positive linear combination of $v_0$ and $u'_0$, which
indicates the existence of the intersectional point $x'_1$. On the other hand, since $u_0$ is also a positive linear combination of $u'_0 $ and $v_1$, by Lemma \ref{lemma-convexity}, we have
$\langle u_0,v_1\rangle_{u_0}^{F_0}
>\langle u'_0 ,v_1\rangle_{u'_0 }^{F_0}$.
Then by Lemma \ref{lemma-geodesic-crossing-edge},
we also have
$$\langle u_1,v_1\rangle_{u_1}^{F_1}
=\langle u_0,v_1\rangle_{u_0}^{F_0}
>\langle u'_0 ,v_1\rangle_{u'_0 }^{F_0}
=\langle u'_1,v_1\rangle_{u'_1 }^{F_1}.
$$
Notice that there exists only two possible cases, namely, either
 $u_1$ is a positive linear combination of $u'_1 $ and $v_1$,
or $u'_1$ is a positive linear combination of $u_1$ and $v_1$.
By Lemma \ref{lemma-convexity}, the second case is impossible. Using this argument inductively, we can prove that $m\leq m'$.

 Now we consider $C_{x,m}$. If $R_{x,v^+}$ is  a ray  intersecting $R_{x_m,u_m}$, then the previous argument can also used to  prove that $R_{x,v^+}$ intersects $L_{x_m,x_{m+1}}$ when $m<m'$,  or $R_{x'_m, u'_m}$ when $m=m'$.

The projection curve
$$\mathrm{Pr}_x^+(\{c_{x_0,w^+}(t),t\in[-1,\infty)\})
=\{\frac{v_{x,c_{x_0,w^+}(t)}}{F(v_{x,c_{x_0,w^+}(t)})}\in
\mathcal{S}_x^+M,t\in[-1,\infty)\}$$
is a monotonous curve in $\mathcal{S}_x^+M$, and so do
the projection curves with respect to $\mathrm{Pr}_x^-$ and for $c_{x_0,w'^+}(t)$.
Above argument shows, counting multiplicities, $\mathrm{Pr}_x^\pm(\{c_{x_0,w^+}(t),t\in[-1,\infty)\})$ are respectively contained in
$\mathrm{Pr}_x^\pm(\{c_{x_0,w'^+}(t),t\in[-1,\infty)\})$,  so we have
$\sphericalangle_{x}^\pm
(\{c_{x_0,w^+}(t),t\in[-1,\infty)\})\leq
    \sphericalangle_{x}^\pm
    (\{c_{x_0,w'{}^+}(t),t\in[-1,\infty)\})$.
The equality holds only  when $m=m'$ and $u_m^+$ is parallel to $u'_m $.
Using Lemma \ref{lemma-geodesic-crossing-edge} repeatedly, we get
$w^+=u_0=u'_0 =w'{}^+$. This is a contradiction. So the inequality is sharp.
The other statement in (1) has  been proved in the previous discussion.

(2) A similar argument as for (1) shows that, counting multiplicities, $\mathrm{Pr}_x^\pm(\{c_{x',w^+}(t)$, $t\in(-1,\infty)\})$ is contained in
$\mathrm{Pr}_x^\pm(\{c_{+,l}(t),t\in(-\infty,\infty)\})$ for any $w^+\in\mathcal{S}_{x'}^{+,l}$.

Now we prove that $\mathrm{Pr}_x^\pm(\{c_{+,l}(t),t\in(-\infty,\infty)\})$ are respectively the limits of $\mathrm{Pr}_x^\pm(\{c_{x',w^+}(t),t\in(-1,\infty)\})$  when $w^+\in\mathcal{S}_{x'}^{+,l}$ approaches $\dot{c}(-1)^+$. By Theorem \ref{lemma-geodesic-not-passing-vertex-x},
  we can write $\{c_{+,l}(t),t\in(-\infty,\infty)\}$ as the
union
$$R_{x_1,u_1^-}\cup(\bigcup_{i=1}^{m-1} L_{x_i,x_{i+1}})\cup R_{x_m, u_m^+},$$
where each $x_i\neq x$ is on the edge $E_i$, and each line segment or ray is not contained in an edge. For $1<i<m$, we denote $u_i=v_{x_i,x_{i+1}}/F_i(v_{x_i,x_{i+1}})$
the unit outgoing vectors at $x_i$ on $L_{x_i,x_{i+1}}$. Since both geodesics $\{c_{+,l}(t),t\in(-\infty,\infty)\}$ and $\{c_{x',w^+}(t),t\in[-1,\infty)\}$ are on the left side of the geodesic $c(t)$ around $x'$, $\{c_{x',w^+}(t),t\in(-1,+\infty)\}$ intersects the edges $E_i$
 sequentially. If $w^+\in\mathcal{S}_{x_0}^{+,l}$ is sufficiently close to $\dot{c}(-1)^+$, then by the continuity implied by the edge-crossing equation (\ref{geodesic-equation-crossing-edge}), the edges that
$\{c_{x',w^+}(t),t\in[-1,\infty)\}$ intersects are exactly all the $m$ edges $E_i$. Let $x'_i$ be the intersection points between $c_{x',w^+}(t)$ and $E_i$, and denote the unit outgoing tangent vectors $u'_i=v_{x'_i,x'_{i+1}}/F_i(v_{x'_i,x'_{i+1}})$
at $x'_i$, for $1\leq i<m$, and $u'_m$ for the Ray
$R_{x'_m,u'_m}$ part of the geodesic $\{c_{x',w^+}(t),t\in[-1,\infty)\}$. When $w^+\in\mathcal{S}_{x_0}^{+,l}$ is sufficiently close to $\dot{c}(-1)^+$, each ${u'_i}$ would be very close to the corresponding $u_i$. This implies that the difference part between  $\mathrm{Pr}^\pm_x(\{c_{x',w^+}(t),t\in(-1,\infty)\})$ and $\mathrm{Pr}^\pm_x(\{c_{+,l}(t),t\in(-\infty,\infty)\})$ respectively have sufficiently small arc length in $\mathcal{S}_x^\pm M$,
when $w^+$ is sufficiently close to $\dot{c}(-1)$.

Combining  the monotonous property we have shown in the proof of (1), with  the fact that $\mathrm{Pr}^\pm_x(\{c_{x',w^+}(t),t\in(-1,\infty)\})$ are open curves in
$\mathcal{S}_x^\pm M$, we complete the proof of the convergence for the sets and for the angles.
The other assertions in (2) also follows immediately.
\end{proof}

Next we consider the perturbation of a unit speed geodesic $\{c(t),t\in[0,\infty)\}$. For $x_0=c(1)$, denote the two connected
components of $\mathcal{S}_{x_0}^-(T_xM)\backslash\{\pm\dot{c}(1)\}$ as
$\mathcal{S}_{x_0}^{-,l}$ and $\mathcal{S}_{x_0}^{-,r}$. Fix
 a local orientation of $M$ at $x$. Then it  defines the left side and
right side. So one can similarly define the geodesics
$\{c_{x_0,w^-}(t),t\in(-\infty,1]\}$ for $w^-\in\mathcal{S}_{x_0}^{-,l}$ or $\mathcal{S}_{x_0}^{-,r}$, and
$\{c_{-,l}(t),t\in(-\infty,\infty)\}$ and $\{c_{-,r}(t),t\in(-\infty,\infty)\}$,  where we add the subscript $-$ to indicate that they are  the extension of a geodesic in the negative direction.

 The following result is an analogue of Lemma \ref{lemma-monotonous}. Since the proof is similar to the previous one,  we  omit it.

\begin{lemma}\label{lemma-monotonous-2}
Let $M$ be a piecewise flat Finsler surface and $T_xM$  the tangent cone of the vertex $x\in M$. Keep all the notations as above. Let $\{c(t),t\in[0,\infty)\}$ be the geodesic with $c(0)=x$ and $c(1)=x_0$. Then  we have the following:
\begin{description}
\item{\rm (1)} If  $w^-, w'^-\in\mathcal{S}_{x_0}^{-,l}$ (or $w^-, w'^-\in\mathcal{S}_{x_0}^{-,r}$) satisfies the condition $$\angle_{x_0}^-(w^-,\dot{c}(1))> \angle_{x_0}^-(w'^-,\dot{c}(1)),$$ then we have
    $$\sphericalangle_{x}^\pm
    (\{c_{x_0,w^-}(t),t\in(-\infty,1]\})<
    \sphericalangle_{x}^\pm
    (\{c_{x_0,w'^-}(t),t\in(-\infty,1]\}).$$
     In particular,
    a ray from $x$ intersects  $\{c_{x_0,w'^-}(t),t\in(-\infty,1]\}$ if
    it intersects  $\{c_{x_0,w^-}(t),t\in(-\infty,1]\}$.
\item{\rm (2)} The angles $\sphericalangle_{x}^\pm
    (\{c_{-,l}(t),t\in(-\infty,\infty)\})$ are respectively the limits of $\sphericalangle_{x}^\pm
    (\{c_{x_0,w^-}(t)$, $t\in(-\infty,1]\})$,  as $w^-\in\mathcal{S}_{x_0}^{-,l}$ approaches  $\dot{c}(1)^-$. In particular, the angle $\sphericalangle_{x}^\pm
    (\{c_{-,l}(t),t\in(-\infty,\infty)\})$ only depends on the side of the parallel shifting. Moreover, a ray from $x$ intersects  $\{c_{-,l}(t),t\in(-\infty,\infty)\}$ iff it intersects some $\{c_{x_0,w^-}(t),t\in(-\infty,1)\}$ with $w^-\in\mathcal{S}_{x_0}^{-,l}$. A similar assertion is also valid for perturbations to the right side.
\end{description}
\end{lemma}

\section{Curvature and extension of geodesics}
\subsection{Definition of curvature}
Now we define the curvature of a piecewise flat Finsler surface $M$ at a vertex $x$. As it is a local geometric quantity,  we only need to define it in the tangent cone $T_xM$.

Let $v^\pm$ be a unit incoming or outgoing tangent vector at $x$.
We have described  several perturbations for the geodesic ray $R_{x,v^\pm}$, i.e.,
the geodesics $c_{x_0,w^\pm}(t)$, $c_{\pm,l}(t)$ and $c_{\pm,r}(t)$.

\begin{definition}\label{def-curvature-piecewise-flat-surface}
Let $M$  be a  piecewise flat Finsler surface and $T_xM$ the tangent cone of a vertex $x\in M$.

(1) For the unit incoming tangent vector $v^-$, the curvature of the tangent cone $T_xM$ in the direction of $v^-$ is
\begin{eqnarray}\label{define-curvature-1}
K(x,v^-)&=&\sphericalangle_x^+
(\{c_{+,l}(t),t\in(-\infty,\infty)\})
+\sphericalangle_x^+(\{c_{+,r}(t),t\in(-\infty,\infty)\})
\nonumber\\
& &-
l_x^+(\mathcal{S}_x^+M).
\end{eqnarray}

(2) For the unit outgoing tangent vector $v^+$ at $x$, the curvature of the tangent cone $T_xM$ in the direction of $v^+$ is
\begin{eqnarray}\label{define-curvature-2}
K(x,v^+)&=&\sphericalangle_x^-
(\{c_{-,l}(t),t\in(-\infty,\infty)\})
+\sphericalangle_x^-(\{c_{-,r}(t),t\in(-\infty,\infty)\})
\nonumber\\& &
-l_x^-(\mathcal{S}_x^-M).
\end{eqnarray}

(3) For a unit incoming or outgoing vector $v^\pm$ at $x\in M$, the curvature $K(x,v^\pm)$ is just the curvature of the tangent cone $T_xM$ in the direction of $v^\pm$ at $x$.
\end{definition}

In the following we first make  several fundamental observations.

First, although we use an orientation on $T_xM$, i.e.,  a local orientation on $M$ to distinguish the left side and right side, the choice of the local orientation does not affect the curvature defined.
 If $x$ is not a vertex, we always have $K(x,v^\pm)=0$
  for all unit tangent vectors (so the curvature of a piecewise flat Finsler
surface only concentrates on the vertices). We may add more edges and vertices
to the triangulation on $M$, which changes a non-vertex point $x$ to a vertex.
It is easy to see that the new vertex has $0$ curvatures for all tangent vectors.

Second, in case that $M$ is a piecewise flat Riemannian surface, the above definition
coincides with
$K=2\pi-\sum\alpha_i$ where $\alpha_i$ is the angle for each cone in $T_xM$;
see \cite{Regge1961}.
Notice that in Riemannian context, the curvature is independent of the unit tangent vector $v^\pm$. However, in Finsler context, it is natural that the curvature depends on a nonzero base vector, e.g., the flag curvature of a general Finsler space.

If $M$ is a reversible, it is not hard to observe that $K(x,v)=K(x,-v)$, where $v$ is an outgoing tangent vector at $x$, and $-v$ is viewed as an incoming tangent vector at $x$.
In the next subsection, we will prove an even stronger statement when $M$ is of Landsberg type, i.e. the curvature at a vertex point $x\in M$ is independent of
the unit incoming or outgoing tangent vector $v^\pm\in\mathcal{S}_x^\pm M$. So it is a generalization for
the curvature form for a smooth Landsberg surface \cite{BCS2000}. This observation will be reverified
in the last section when we prove the combinatoric Gauss-Bonnet
formula.

\subsection{Curvature of a piecewise flat Landsberg surface}

In this subsection,
We  study  the curvature property of a piecewise flat Landsberg surface.

\begin{theorem} \label{main-thm-3}
Let $M$ be a piecewise flat Finsler surface  of Landsberg type. Then
the curvature at a vertex $x$ is independent of  $v^\pm\in\mathcal{S}_x^\pm M$.
\end{theorem}

\begin{proof} For the convenience, we fix a local orientation at $x\in M$, i.e.,  we fix an orientation on the tangent cone $T_xM$.

We first prove that $K(x,v^-)$ is independent of $v^-\in\mathcal{S}_x^- M$. The discussion can be carried out in the tangent cone $T_xM$. For any $v_0^-, v_1^-\in\mathcal{S}_x^-M$, there exists a piecewise smooth monotonous family of unit vectors $v_s^-\in\mathcal{S}_x^-M$, $s\in[0,1]$,  connecting them. With respect to the chosen orientation, there are perturbation geodesics $c_{+,l}(t)$ and $c_{+,r}(t)$ for each ray $R_{x,v_s^-}$, which are denoted as $\{c_{s,+,l}(t),t\in(-\infty,\infty)\}$ and $\{c_{s,+,r}(t),t\in(-\infty,\infty)\}$, respectively.

 Now for any $s\in[0,1]$,  $\{c_{s,+,l}(t),t\in(-\infty,\infty)\}$ is a finite union $$R_{x_{s,1},u_{s,1}^-}\cup (\bigcup_{i=1}^{m_s-1} L_{x_{s,i},x_{s,i+1}})\cup R_{x_{s,m_s},u_{s,m_s}^+},$$
where each $x_{s,i}=c_{s,+,l}(t_{s,i})\neq x$ lies on the edge $E_i$. Note that the rays and line segments are not contained in edges, but are contained in the Minkowski cones $(C_{x,i},F_i)$,
$i=0,1,\ldots,m_s$, respectively. Moreover, we have the unit tangent vectors $u_{s,i}^\pm=\dot{c}_{s,+,l}^\pm
(t_{s,i})\in\mathcal{S}_{x_{i}}^+(T_xM)$, $u_{s,1}^-=v_{s}^-$, and
$u_{s,i}^+=u_{s,i+1}^-$ (because they are in the same direction of $L_{x_{s,i},x_{s,i+1}}$). Denote the edges of $C_{x,i}$ as $E_{i}$ and $E_{i+1}$, and for $1\leq i\leq m_s$, we have $E_i=C_{x,i-1}\cap C_{x,i}$.

The curvature $K(x,v_s^-)$ depends continuously on the parameter $s$.
For almost all $s\in[0,1]$, none of $u_{s,i}^\pm$ is parallel to any edge. Let $s_0\in[0,1]$ be such a generic parameter. Then
for $s$ sufficiently close to $s_0$, $m_s$ is a constant $m$, i.e.,
$$\{c_{s,+,l}(t),t\in(-\infty,\infty)\}=
R_{x_{s,1},u_{s,1}^-}\cup (\bigcup_{i=1}^{m-1} L_{x_{s,i},x_{s,i+1}})\cup R_{x_{s,m},u_{s,m}^+},
$$
for all $s\in(s_0-\epsilon,s_0+\epsilon)$ with sufficiently small $\epsilon>0$.

By the Landsberg condition, we have
\begin{eqnarray}\label{2003}
\langle\frac{d}{ds}v_{s}^-|_{s=s_0},
\frac{d}{ds}v_{s}^-|_{s=s_0}
\rangle^{F_0}_{v_{s_0}^-} &=&
\langle\frac{d}{ds}u_{1,s}^-|_{s=s_0},
\frac{d}{ds}u_{1,s}^-|_{s=s_0}
\rangle^{F_1}_{u_{1,s}^-}\nonumber\\
&=&\langle\frac{d}{ds}u_{1,s}^+|_{s=s_0},
\frac{d}{ds}u_{1,s}^+|_{s=s_0}
\rangle_{u_{1,s}^+}^{F_1}\nonumber\\
&=&\langle\frac{d}{ds}u_{2,s}^-|_{s=s_0},
\frac{d}{ds}u_{2,s}^-|_{s=s_0}
\rangle_{u_{2,s}^-}^{F_2}\nonumber\\
&=&\cdots\nonumber\\
&=&\langle\frac{d}{ds}u_{m,s}^+|_{s=s_0},
\frac{d}{ds}u_{m,s}^+|_{s=s_0}
\rangle_{u_{m,s}^+}^{F_m}.
\end{eqnarray}
By (\ref{2003}) and a similar equality for $c_{s,+,r}(t)$,
we get $\frac{d}{ds}K(x,v_s^-)|_{s=s_0}=0$. Combining this fact with the continuity of $K(x,v_s^-)$ for the parameter $s$,
we conclude that $K(x,v_s^-)\equiv \mathrm{const}$, for $s\in[0,1]$, i.e., $K(x,v^-)$
is independent of $v^-\in\mathcal{S}_x^-M$.

A similar argument shows that $K(x,v^+)$ is independent of $v^+\in
\mathcal{S}_x^+M$.

Above  argument can also be used to prove

\bigskip
{\bf Assertion (C).} For any point $x$ on a piecewise flat Landsberg surface, we have
$l_x^+(\mathcal{S}_x^+M)=l_x^-(\mathcal{S}_x^-M)$.
\bigskip

Let $v_s^-$ move monotonously around the whole circle $\mathcal{S}_x^-M$. Then with respect  to the chosen orientation in $T_xM$, the ray in $\{c_{s,+,l}(t),t\in(-\infty,\infty)\}$ for $t$ going to $\infty$ defines
a vector $v_s^+\in\mathcal{S}_x^+M$ which also moves monotonously around the whole circle $\mathcal{S}_x^+M$. By (\ref{2003}),  we  get $l_x^+(\mathcal{S}_x^+M)=l_x^-(\mathcal{S}_x^-M)$, which proves the assertion.

Finally, we prove that $K(x,v^-)=K(x,v^+)$ for any $v^\pm\in\mathcal{S}_x^\pm M$.

Consider a vector $v_1^+\in\mathcal{S}_x^+M$.  There exists a non-negative integer $m$, such that
\begin{equation}\label{0005}
(m-1)\cdot l_x^+(\mathcal{S}_x^+M)<K(x,v_1^+)\leq m\cdot l_x^+(\mathcal{S}_x^+M).
\end{equation}
Assume $m=0$, i.e.,  $K(x,v_1^+)\leq 0$.
We have two geodesics
$\{c_{1,-,l}(t),t\in(-\infty,\infty)\}$ and $\{c_{1,-,r}(t),t\in(-\infty,\infty)\}$ which define $v_0^-,v_1^-\in\mathcal{S}_x^-M$ respectively by their rays when $t$ goes to $-\infty$.
Consider a piecewise smooth family of vectors $v_s^-\in\mathcal{S}_x^-M$,  $s\in[0,1]$, connecting $v_0^-$ and $v_1^-$, which rotates around $\mathcal{S}_{x}^-M$ anti-clockwise (with respect to the given orientation), but does not cover the whole circle $\mathcal{S}_{x}^-M$. Then for any $v_s^-$, we have the geodesic $c_{s,+,r}(t)$ which defines $v_s^+\in\mathcal{S}_x^+M$ by the ray for $t$ going to $+\infty$. Notice that $v_1^+$ coincides  the one we are using. By a similar equality as (\ref{2003}), we have
\begin{equation}\label{0004}
l_x^-(\{v_s^-,s\in[0,1]\})
=l_x^+(\{v_s^+,s\in[0,1]\}).
\end{equation}
The outgoing tangent vectors $v_s^+$ at $x$ for $s\in[0,1]$ also rotate anti-clockwise, and they can not cover the whole circle $\mathcal{S}_x^+M$.

Notice that the geodesics $c_{1,-,l}(t)$ and $c_{0,+,r}(t)$ can also be used to define the curvature $K(x,v_0^-)$, i.e.,
\begin{eqnarray*}
K(x,v_0^-)&=&\sphericalangle_x^+
(\{c_{1,-,l}(t),t\in(-\infty,\infty)\})
+\sphericalangle_x^+(\{c_{0,+,r}(t),t\in(-\infty,\infty)\})-
l_x^+(\mathcal{S}_x^+M)\\
&=&-l_x^+(\{v_s^+,s\in[0,1]\})=-l_x^-(\{v_s^-,s\in[0,1]\})\\
&=&\sphericalangle_x^-(\{c_{1,-,l}(t),t\in(-\infty,\infty)\})
+\sphericalangle_x^-(\{c_{1,-,r}(t),t\in(-\infty,\infty)\})
-l_x^-(\mathcal{S}_x^-M)\\
&=&K(x,v_1^+).
\end{eqnarray*}
This finish  the proof when the curvature is non-positive.

  In the case that $K(x,v_1^+)>0$, i.e.,  the integer $m$ in (\ref{0005}) is positive, the argument is almost the same. We still have $v_0^-$ and $v_1^-$ defined
from the rays in the geodesics $c_{-,l}(t)$ and $c_{-,r}(t)$ when $t$ goes to
$-\infty$. We
just need to change the monotonous family $\{v_s^-,s\in[0,1]\}$ in $\mathcal{S}_x^-M$, such that when $s$ goes from 0 to 1, they
rotate clockwise around the whole circle $\mathcal{S}_x^-M$ for $m-1$ times (and reach some part of it for one more time). Then so do the family $\{v_s^+,s\in[0,1]\}$ in $\mathcal{S}_x^+M$, and we have
\begin{eqnarray*}
K(x,v_0^-)
=l_x^+(\{v_s^+,s\in[0,1]\})=l_x^-(\{v_s^-,s\in[0,1]\})
=K(x,v_1^+),
\end{eqnarray*}
which completes the proof of the theorem.
\end{proof}
\begin{remark}
 Based on this theorem, we can simply denote the curvature of a piecewise flat Landsberg surface as $K(x)$.
\end{remark}

Also we restate Assertion (C) as the following proposition.
\begin{proposition}\label{lemma-5}
For any point $x$ on a piecewise flat Landsberg surface, we have
$l_x^+(\mathcal{S}_x^+M)=l_x^-(\mathcal{S}_x^-M)$.
\end{proposition}

\subsection{The answer to Question \ref{question}}
Now we are ready to answer Question \ref{question}.
For the extension of a geodesic in the positive direction, we have
\begin{theorem} \label{main-thm-1}
 Let $M$ be a piecewise flat Finsler surface and $T_xM$ be the tangent cone at a vertex $x\in M$.
Let $\{c(t),t\in(-\infty,0]\}$ be a unit speed geodesic in the tangent cone $T_xM$ with $c(0)=x$. Then the following statements are equivalent:
\begin{description}
\item{\rm (1)} $c(t)$ can be extended to a geodesic defined on  $(-\infty,\infty)$.
\item{\rm (2)} $K(x,\dot{c}^-(0))\leq 0$.
\item{\rm (3)} There exists a ray $R_{x,w^+}$ such that $w^+\neq -\dot{c}^-(0)$ and  it does not intersect  any
    $\{c_{+,l}(t),t\in(-\infty,\infty)\}$ or
    $\{c_{+,r}(t),t\in(-\infty,\infty)\}$.
\item{\rm (4)} There exists a ray $R_{x,w^+}$ which does not intersect
 any of the geodesics $\{c_{x_0,w'^+}(t),$ $t\in[-1,\infty)\}$ with $c(-1)=x_0$,
and $w'^+\in\mathcal{S}_{x_0}^{+,l}$ or $\mathcal{S}_{x_0}^{+,r}$.
\end{description}
\end{theorem}

\begin{proof}
The equivalence between (3) and (4) has been established  in  Lemma \ref{lemma-monotonous}.

\textbf{$(2)\Rightarrow (3)$}\quad
If $K(x,\dot{c}(0)^-)\leq 0$, then by Definition \ref{define-curvature-1},  $\mathrm{Pr}_x^+\{c_{+,l}(t),t\in(-\infty,\infty)\}$
and $\mathrm{Pr}_x^-\{c_{+,r}(t),t\in(-\infty,\infty)\}$ are disconnected
open curves in the circle $\mathcal{S}_x^+M$, i.e.,  there exists a vector
$w^+\neq-\dot{c}(0)^-$ in
$$\mathcal{S}_x^+M\backslash
(\mathrm{Pr}_x^+\{c_{+,l}(t),t\in(-\infty,\infty)\}
\cup\mathrm{Pr}_x^-\{c_{+,r}(t),t\in(-\infty,\infty)\}).$$
Notice that $\mathrm{Pr}_x^+\{c_{+,l}(t),t\in(-\infty,\infty)\}$ and $\mathrm{Pr}_x^-\{c_{+,r}(t),t\in(-\infty,\infty)\}$ are independent of the particular choice of the perturbations whenever the side has been chosen. So
the ray $R_{x,w^+}$ has no intersection with any geodesics $\{c_{+,l}(t),t\in(-\infty,\infty)\}$ and
$\{c_{+,r}(t),t\in(-\infty,\infty)$.

$(3)\Rightarrow (2)$\quad  It is easily seen that the outgoing tangent vector $w^+\neq-\dot{c}(0)^-$ at $x$ is  contained in
$$\mathcal{S}_x^+M\backslash(\mathrm{Pr}_x^+\{c_{+,l}(t),t\in(-\infty,\infty)\}
\cup\mathrm{Pr}_x^-\{c_{+,r}(t),t\in(-\infty,\infty)\}),$$
and $-\dot{c}(0)^-$ is a common end point of
$\mathrm{Pr}_x^+\{c_{+,l}(t),t\in(-\infty,\infty)\}$ and
$\mathrm{Pr}_x^-\{c_{+,r}(t),t\in(-\infty,\infty)\}$. So
$$\mathrm{Pr}_x^+\{c_{+,l}(t),t\in(-\infty,\infty)\}
\cap\mathrm{Pr}_x^-\{c_{+,r}(t),t\in(-\infty,\infty)\}=\emptyset,$$
which implies that
$$K(x,v^-)=\sphericalangle_x^+(\{c_{+,l}(t),t\in(-\infty,\infty)\})
+\sphericalangle_x^+(\{c_{+,r}(t),t\in(-\infty,\infty)\})-
l_x^+(\mathcal{S}_x^+M)\leq 0.$$

 $(1)\Rightarrow (4)$\quad   Let $\{c(t),t\in(-\infty,\infty)$ be the unit speed geodesic consisting of two rays. Assume conversely that there exists a
 geodesic $\{c_{x_0,w'^+}(t),t\in[-1,\infty)\}$, where
$c(-1)=c_{x_0,w'^+}(-1)=x_0$,
$w'^+\in\mathcal{S}_{x_0}^+M\backslash \{\pm\dot{c}(-1)^+\}$,
and $c(t')=c_{x_0,w'^+}(t'')=x_1$ with $t'>0$. Then by Lemma \ref{lemma-2},
we have $l(\{c_{x_0,w'^+}(t),t\in[-1,t'']\})<l(\{c(t),t\in[-1,t']\})$.
So $c(t)$ is not a minimizing geodesic from $x_0$ to $x_1$. This is a contradiction
with Lemma \ref{lemma-global-minimizing}.

$(4)\Rightarrow (1)$\quad By Theorem \ref{lemma-local-minimizing} below,  for any two different points $x'$ and $x''$ in the tangent cone $T_xM$,
there exists a minimizing geodesic connecting  $x'$ and $x''$.
Let $x'=c(-1)$ and $x''\neq x$ be any point on the ray indicated in (4). Then by
Lemma \ref{lemma-local-minimizing}, and our assumptions in (4), the  minimizing geodesic from $x'$ to $x''$ must be the union of $L_{x',x}$ and $L_{x,x''}$, i.e.,  $c(t)$ can be extended to a geodesic consisting of two rays.

This completes the proof of Theorem \ref{main-thm-1}.
\end{proof}

For the extension of a geodesic in the negative direction, we have

\begin{theorem} \label{main-thm-2}
 Let $M$ be a piecewise flat Finsler surface and $T_xM$ be the tangent cone at a vertex $x\in M$.
Let $\{c(t),t\in[0,-\infty)\}$ be a unit speed geodesic in the tangent cone $T_xM$ with $c(0)=x$. Then the following statements are equivalent:
\begin{description}
\item{\rm (1)} $c(t)$ can be extended to a geodesic defined on $(-\infty,\infty)$.
\item{\rm (2)} $K(x,\dot{c}(0)^+)\leq 0$.
\item{\rm (3)} There exists a ray $R_{x,w^-}$ such that $w^-\neq -\dot{c}(0)^+$ and it does not   intersect  any
    $\{c_{-,l}(t),t\in(-\infty,\infty)\}$ or
    $\{c_{-,r}(t),t\in(-\infty,\infty)\}$.
\item{\rm (4)} There exists a ray $R_{x,w^-}$ which does not  intersect
 any of the geodesics $\{c_{x_0,w'^-}(t)$, $t\in(-\infty,1]\}$,  where $c(1)=x_0$,
and $w'^+\in\mathcal{S}_{x_0}^{-,l}$ or $\mathcal{S}_{x_0}^{-,r}$.
\end{description}
\end{theorem}

The proof is similar to the previous one, so we omit it.

The following is the restatement of
Theorem \ref{main-thm-1} and \ref{main-thm-2} in the context of $M$ instead of $T_xM$.
\begin{corollary}
 Let $M$  be a piecewise flat Finsler surface and $T_xM$ be the tangent cone at a vertex $x\in M$. Then a unit speed geodesic
$\{c(t), t\in(-\epsilon,0]$ (resp. $\{c(t),t\in[0,\epsilon)\}$) on $M$ with $c(0)=x$
can be extended at $x$ if and only if $K(x,\dot{c}(0)^-)\leq 0$ (resp. $K(x,\dot{c}(0)^-)\geq 0$).
\end{corollary}

From the proof of Theorem \ref{main-thm-1} and \ref{main-thm-2}, we also get the following

\begin{corollary}
 Let $M$ be a  piecewise flat Finsler surface  and $T_xM$  the tangent cone at the vertex $x\in M$.
 If the rays  in (3) and (4) of Theorem \ref{main-thm-1} and \ref{main-thm-2} exist, then they exhaust  all the extensions of the geodesic $c(t)$ for $t\in(-\infty,\infty)$. When $K(x,v^\pm)<0$, there
are infinitely many such extensions. When $K(x,v^\pm)=0$, there exists a unique
such extension. When $K(x,v^\pm)>0$, there exists no such extensions.
\end{corollary}

 This also gives an interpretation to Lemma \ref{geodesic-extension-crossing-edge}.
 In fact,   if $x$ is not a vertex, then the curvature is always $0$. Thus the extension of a geodesic at $x$  exists uniquely.
The notions of curvature we  define in this paper plays the same role as their counter part in the Riemannian case, i.e.,  for $v^\pm\in\mathcal{S}_x^\pm M$, $-K(x,v^\pm)$ is
the measure for all the geodesics extending $R_{x,v^\pm}$ at $x$. When $K(x,v^\pm)>0$, the virtual measure for such extensions is negative. It
explains why there exists no such extensions in this case.

At the final part of this section, we prove a general result on the existence of minimizing geodesic.

\begin{theorem} \label{lemma-local-minimizing}
For any two different points $x'$ and $x''$ in the tangent cone $T_xM$,
there exists a minimizing geodesic from $x'$ to $x''$.
\end{theorem}

\begin{proof}
When $x$ is in the complement of edges and vertices, the theorem is obvious.
When $x$ is an edge point but not a vertex,
it has been proven by Theorem \ref{theorem-000}.
So we may assume $x$ is a vertex, and then $T_xM$ contains at least 3 Minkowski cones.

First we consider the case $x'=x$. We claim that $L_{x,x''}$ is the minimizing geodesic from $x'$ to $x''$. Assume conversely that this is not true.
Then there  exists a shorter path $\{p(t),x,x''\}$ from $x$ to $x''$, consisting
of a finite sequence of straight line segments. Using the triangular inequality repeatedly, we will get $l(\{p(t),x,x''\})\geq l(L_{x,x''})$, which is a contradiction. The argument for the case $x''=x$ is similar.

Now we assume $x'\neq x$ and $x''\neq x$.
If $x'$ and $x''$ are contained in the same Minkowski cone such that
$d(x',x'')=l(L_{x',x''})$, then $L_{x',x''}$ is the minimizing geodesic
from $x'$ to $x''$.
If $d(x',x'')=l(L_{x',x})+l(L_{x,x''})$, then the two straight line segments from $x'$ to $x$, and  from $x$ to $x''$ gives the minimizing geodesic from $x'$ to $x''$.

 If the minimizing geodesic from $x'$ to $x''$ can not be realized as above, then there exists $c>0$
satisfying the following conditions:
\begin{description}
\item{\rm (1)} If $x'$ and $x''$ are contained in the same Minkowski cone, then $d(x',x'')+c<l(L_{x',x''})$.
\item{\rm (2)} $c<d(x',x)+d(x,x'')-d(x',x'')$.
\end{description}
Then we can find a sequence of piecewise linear paths $\{p_n(t),x',x''\}$ from $x'$ to $x''$, $n\in\mathbb{N}$, such that each path consists of a finite sequence of straight line segments, and its arc length satisfies
\begin{equation}\label{0002}
l(\{p_n(t),x',x''\})<d(x',x'')+\frac{c}{2n}.
\end{equation}
Then $l(\{p_n(t),x',x''\})<d(x',x)+d(x,x'')$,
and $l(\{p_n(t),x',x''\})<l(L_{x',x''})$  when $x'$ and $x''$ are in the same cone.

Using the triangular inequality repeatedly, we can make $\{p_n(t),x',x''\}$ even shorter by assuming the following:
\begin{description}
\item{\rm (1)} $\{p_n(t),x',x''\}=
    L_{x',x_{n,1}}\cup
    (\bigcup_{i=1}^{m_n-1}L_{x_{n,i},x_{n,i+1}})\cup
    L_{x_{n,m_n},x''}$.
\item{\rm (2)} The points $x_{n,1}$, $\ldots$, $x_{n,m_n}$ are on the edges.
\end{description}

If some $x_{n,i}$ is sufficiently close to $x$, say,  $d(x_{n,i},x)+d(x,x_{n,i})<c/6$, then we have
\begin{eqnarray*}
d(x',x'')+c &<& d(x',x)+d(x,x'')\leq l(\{p_n(t),x',x''\})+d(x_{n,i},x)+d(x,x_{n,i})\\
&<& d(x',x'')+\frac{5}{6}c,
\end{eqnarray*}
which is a contradiction.
 On the other hand, each $x_{n,i}$ can not be too far
away from $x$ either, because
\begin{eqnarray*}
d(x_{n,i},x)+d(x,x_{n,i})&\leq& l(\{p_n(t),x',x''\})+d(x,x')+d(x'',x)\\
&\leq& d(x',x)+d(x,x')+d(x'',x)+d(x,x'').
\end{eqnarray*}
If there exists a edge containing at least one of  the pairs $(x_{n,i},x_{n,i+1})$,   $(x', x_{n,1})$, or  $(x_{n,m_n}, x'')$, then
by  the triangular inequality there exists  a shorter path $p_n(t)$ (in this case the number $m_n$ will decrease at the same time). Notice that
our assumptions for the constant $c$ and the path $p_n(t)$ implies that
$x'$, $x''$, and the $x_{n,i}$'s can not be contained in the same edge.
To summarize, we can further require $p_n(t)$ to satisfy
\begin{description}
\item{\rm (3)} For each $i$, $c_1<d(x_{n,i},x)+d(x,x_{n,i})<c_2$, where $c_1$ and
$c_2$ are positive numbers independent of $n$ and $i$.
\item{\rm (4)} None of the straight line segments $L_{x',x_{n,1}}$, $L_{x_{n,1},x_{n,2}}$, $\ldots$, $L_{x_{n,m_n-1},x_{n,m_n}}$ and
    $L_{x_{n,m_n},x''}$ is contained in any edge.
\end{description}
By (3), there exists a constant $c_3>0$, independent of $n$ and $i$, such   that $d(x_{n,i},x_{n,i+1})>c_3$. So by (2), there exists  a universal bound $N>0$ such that $m_n<N$ for all $n$.

 Therefore, passing to suitable sub-sequence $\{p_n(t),x',x''\}$ with the same $m_n=m$ if necessary, we can get $\mathop{\lim}\limits_{n\rightarrow\infty} x_{n,i}=x_i$ simultaneously, $\forall 1\leq i\leq m$. Note that  all the paths $p_n(t)$ are parametrized by the arc length. Hence   $p(t)=\lim p_n(t)$ is a finite union of line segement, which
defines a minimizing geodesic from $x'$ to $x''$.
\end{proof}

\section{A combinational Gauss-Bonnet formula for connected piecewise flat Landsberg surface}

In this section, we assume $M$ is a connected piecewise flat Landsberg surface. The next
lemma follows immediately from the connectedness of $M$ and the Landsberg condition.

\begin{lemma} \label{lemma-theta-M} Let $M$ be a connected piecewise flat Landsberg surface. Then for any $x\in M$ which is not a vertex, $l_x^\pm(\mathcal{S}_x^\pm M)$ is a  constant.
\end{lemma}
 In the following, we will denote the above constant as $\theta_M$.

Our goal in this section is to prove a combinatoric
Gauss-Bonnet formula.

\begin{theorem}\label{main-thm-4}
Let $M$ be a compact connected piecewise flat Landsberg surface (without boundary).
Then
\begin{equation}\label{gauss-bonnet-formula}
\sum_{x\in M} K(x)=\theta_M\chi(M),
\end{equation}
where $\theta_M$ is the constant in Lemma \ref{lemma-theta-M} and
$\chi(M)$ is the Euler characteristic number.
\end{theorem}

To prove Theorem \ref{main-thm-4}, we first make some
observations.

\medskip
\noindent (1)\quad One can  add more vertices and edges for the triangulation of $M$. In fact, the curvature at a new vertex $x$ is 0. The curvature at an old vertex $x$ will not be changed by adding more edges associated with it. So the left side of (\ref{gauss-bonnet-formula}) is not changed by this procedure.

\medskip
\noindent(2)\quad One can suitably add some vertices and edges such that the following conditions are satisfied:
\begin{description}
\item{\rm (a)} The total number of the vertices is even.
\item{\rm (b)} The vertices can be listed as $\{z_i,i=1,\ldots,2m\}$
such that for each $1\leq i\leq m$, there exists an edge between $z_{2i-1}$
and $z_{2i}$.
\end{description}
To see this, we consider a vertex $z_1$ in each step. Assume there exists
a Minkowski triangle in $M$ with vertices $z_1$, $z_2$ and $z_3$. We then add
a vertex $z_4$ at the geometric center of the this triangle, and add three edges connecting $z_4$ to $z_1$, $z_2$ and $z_3$, respectively.  Pair $z_1$ with $z_4$. Then we continue to consider other vertices. In each step, the number of vertices which has not been paired decrease by 1. So after finite steps, all the vertices are paired, and they can be listed as in (2) such that for each $i$, $z_{2i-1}$ and $z_{2i}$ have been paired together.

\medskip
\noindent(3)\quad Consider a geodesic $\{c(t),t\in(-\infty,\infty)\}$ in $T_xM$ which does not pass the vertex $x$. Assume this geodesic can be presented as
$$R_{x_1,u_1^-}\cup(\bigcup_{i=1}^{m-1} L_{x_i,x_{i+1}})\cup R_{x_m,u_m^+},$$
where each $x_i\neq x$ is on the edge $E_i=C_{x,i-1}\cap C_{x,i}$, and the rays
and line segments are contained in the Minkowski cones $C_{x,i}$, $i=0,\ldots,m$,
respectively. Denote $u_i^\pm$ with $u_i^+=u_{i+1}^-$ the unit outgoing or incoming tangent vector
of $c(t)$ at each $x_i$, and $v_i$ the unit tangent vector in the direction from $x$ to $x_i$.
Then we have
\begin{lemma}\label{lemma-3}
$\mathrm{1)}$ $\sphericalangle_x^+(\{c(t),t\in(-\infty,\infty)\})
=\angle_{x_1}^+(-u_1^-,v_1)
+\angle_{x_1}^-(u_1^-,v_1)$.

$\mathrm{2)}$ $\sphericalangle_x^-(\{c(t),t\in(-\infty,\infty)\})
=\angle_{x_m}^-(-u_m^+,-v_m)
+\angle_{x_m}^+(u_m^+,-v_m)$.
\end{lemma}

\begin{proof}
1) By the Landsberg condition, we have
\begin{eqnarray*}
\angle_{x_1}^-(u_1^-,v_1)&=&
\angle_{x_1}^+(u_1^+,v_1)\\
&=&\angle_{x}^+(v_1,v_2)+\angle_{x_2}^-(u_2^-,v_2)\\
&=&\angle_{x}^+(v_1,v_2)+\angle_{x_2}^+(u_2^+,v_2)\\
&=&\angle_{x}^+(v_1,v_2)+\angle_x^+(v_2,v_3)+
\angle_{x_3}^-(u_3^-,v_3)\\
&=&\cdots\\
&=&\sum_{i=1}^{m-1}\angle_x^+
(v_i,v_{i+1})+\angle_{x_m}^+(v_m,u_m^+).
\end{eqnarray*}
Thus
\begin{eqnarray*}
\sphericalangle_x^+(\{c(t),t\in(-\infty,\infty)\})&=&
\angle_{x_1}^+(-u_1^-,v_1)+
\sum_{i=1}^{m-1}\angle_x^+(v_i,v_{i+1})+
\angle_{x_m}^+(v_m,u_m^+)\\
&=&
\angle_{x_1}^+(-u_1^-,v_1)+\angle_{x_1}^-(u_1^-,v_1).
\end{eqnarray*}

2) The proof is similar.
\end{proof}

\medskip
\noindent(4)\quad
In each Minkowski triangle with vertices $z_1$, $z_2$ and $z_3$,
we have
\begin{eqnarray}\label{0006}
\theta_M&=&\sphericalangle_{z_{1}}^+(L_{z_{2},z_{3}})
+
\sphericalangle_{z_{1}}^-(L_{z_{2},z_{3}})+
\sphericalangle_{z_{2}}^+(L_{z_{1},z_{3}})+
\sphericalangle_{z_{2}}^-(L_{z_{1},z_{3}})\nonumber\\
& &+
\sphericalangle_{z_{3}}^+(L_{z_{1},z_{2}})+
\sphericalangle_{z_{3}}^-(L_{z_{1},z_{2}})\nonumber\\
&=&
\angle_{z_1}^+(v_{z_1,z_2},v_{z_1,z_3})
+\angle_{z_1}^-(v_{z_2,z_1},v_{z_3,z_1})+
\angle_{z_2}^+(v_{z_2,z_1},v_{z_2,z_3})\nonumber\\
& &+\angle_{z_2}^-(v_{z_1,z_2},v_{z_3,z_2})
+
\angle_{z_3}^+(v_{z_3,z_1},v_{z_3,z_2})+
\angle_{z_3}^-(v_{z_1,z_3},v_{z_2,z_3}).
\end{eqnarray}
Moreover, for any $z_4\notin L_{z_2,z_1}$ on the ray extending $L_{z_2,z_1}$, we have
\begin{equation}\label{0007}
\angle_{z_1}^-(v_{z_4,z_1},v_{z_3,z_1})
=\angle_{z_2}^-(v_{z_1,z_2},v_{z_3,z_2})
+\angle_{z_3}^+(v_{z_3,z_1},v_{z_3,z_2}),
\end{equation}
and
\begin{equation}\label{0008}
\angle_{z_1}^+(v_{z_1,z_4},v_{z_1,z_3})=
\angle_{z_2}^+(v_{z_2,z_1},v_{z_2,z_3})
+\angle_{z_3}^-(v_{z_1,z_3},v_{z_2,z_3}).
\end{equation}
These obvious facts are natural generalizations of fundamental facts in classical geometry, and we will omit the proof.

\medskip
\noindent(5)\quad Consider two Minkowski triangles $(T_1,F_1)$ and $(T_2,F_2)$
in $M$ whose vertices are $\{z_1,z_2,z_3\}$ and $\{z_1,z_2,z_4\}$, respectively. Denote the curvature at $z_1$ and $z_2$ as $K(z_1)$ and $K(z_2)$, respectively. Then we have
\begin{lemma}\label{lemma-4}
$K(z_1)+K(z_2)=2\theta_M-l_{z_1}^\pm(\mathcal{S}_{z_1}^\pm M)
-l_{z_2}^\pm(\mathcal{S}_{z_2}^\pm M)$.
\end{lemma}

Notice that by Proposition \ref{lemma-5},
at each vertex $z_i$, we have $l_{z_i}^+(\mathcal{S}_{z_i}^+M)=l_{z_i}^-(\mathcal{S}_{z_i}^-M)$.

\begin{proof}
 Fix two geodesics $c_1(t)$ and $c_2(t)$  in $T_1$ and $T_2$  which are sufficiently close and parallel to the edge $L_{z_1,z_2}$, respectively. Assume the geodesic $c_1(t)$ intersects with $L_{z_2,z_3}$ at $x_1$, and with $L_{z_1,z_3}$ at $x_2$. Suppose  the geodesic $c_2(t)$ intersects with $L_{z_2,z_4}$ at $x_3$, and with $L_{z_1,z_4}$ at $x_4$. Moreover, we assume   that the direction of $c_1(t)$ is from $x_1$ to $x_2$, and that of $c_2(t)$ is from $x_3$ to $x_4$.
Then
by Theorem \ref{main-thm-3}, Lemma \ref{lemma-3}, and (\ref{0006})-(\ref{0008}), we can get
\begin{eqnarray*}
& &K(z_1)+K(z_2)\\
&=&
\angle_{x_1}^-(v_{z_3,x_1},v_{x_2,x_1}) +
\angle_{x_1}^+(v_{x_1,z_2},v_{x_1,x_2}) +
\angle_{x_3}^-(v_{z_4,x_3},v_{x_4,x_3}) +
\angle_{x_3}^+(v_{x_3,z_2},v_{x_3,x_4}) \\
& &-l_{z_2}^-(\mathcal{S}_{z_2}^- M)\\
& &\angle_{x_2}^+(v_{x_2,z_3},v_{x_2,x_1}) +
\angle_{x_2}^-(v_{z_1,x_2},v_{x_1,x_2}) +
\angle_{x_4}^+(v_{x_4,z_4},v_{x_4,x_3}) +
\angle_{x_4}^-(v_{z_1,x_4},v_{x_3,x_4}) \\
& &-l_{z_1}^+(\mathcal{S}_{z_1}^+ M) \\
&=&
\angle_{x_1}^-(v_{z_3,x_1},v_{x_2,x_1}) +
(\angle_{x_2}^-(v_{z_3,x_2},v_{x_1,x_2}) +
\angle_{z_3}^+(v_{z_3,x_1},v_{z_3,x_2})) +
\angle_{x_3}^-(v_{z_4,x_3},v_{x_4,x_3})\\
& &+
(\angle_{x_4}^-(v_{z_4,x_4},v_{x_3,x_4}) +
\angle_{z_4}^+(v_{z_4,x_3},v_{z_4,x_4})) +
\angle_{x_2}^+(v_{x_2,z_3},v_{x_2,x_1}) +
(\angle_{x_1}^+(v_{x_1,z_3},v_{x_1,x_2})\\
& &+\angle_{z_3}^-(v_{x_2,z_3},v_{x_1,z_3}))+
\angle_{x_4}^+(v_{x_4,z_4},v_{x_4,x_3}) +
(\angle_{x_3}^+(v_{x_3,z_4},v_{x_3,x_4})+
\angle_{z_4}^-(v_{z_4,x_4},v_{z_4,x_3}))\\
& &-l_{z_1}^\pm(\mathcal{S}_{z_1}^\pm M) -
l_{z_2}^{\pm}(\mathcal{S}_{z_2}^\pm M) \\
&=&
(\angle_{z_1}^+(v_{z_1,z_3},v_{z_1,z_2}) +
\angle_{z_1}^-(v_{z_3,z_1},v_{z_2,z_1}) +
\angle_{z_2}^+(v_{z_2,z_3},v_{z_2,z_1}) +
\angle_{z_2}^-(v_{z_3,z_2},v_{z_1,z_2})\\
& & +
\angle_{z_3}^+(v_{z_3,z_2},v_{z_3,z_1}) +
\angle_{z_3}^-(v_{z_1,z_3},v_{z_2,z_3})) +
(\angle_{z_1}^+(v_{z_1,z_4},v_{z_1,z_2}) +
\angle_{z_1}^-(v_{z_4,z_1},v_{z_2,z_1}) \\
& &+
\angle_{z_2}^+(v_{z_2,z_4},v_{z_2,z_1}) +
\angle_{z_2}^-(v_{z_4,z_2},v_{z_1,z_2}) +
\angle_{z_4}^+(v_{z_4,z_2},v_{z_4,z_1}) +
\angle_{z_4}^-(v_{z_4,z_1},v_{z_4,z_2})) \\
& &-l_{z_1}^\pm(\mathcal{S}_{z_1}^\pm M) -
l_{z_2}^{\pm}(\mathcal{S}_{z_2}^\pm M) \\
&=&
2\theta_M-l_{z_1}^\pm(\mathcal{S}_{z_1}^\pm M) -
l_{z_2}^{\pm}(\mathcal{S}_{z_2}^\pm M)
\end{eqnarray*}

This completes the proof of the lemma.
\end{proof}

\medskip
\noindent\textbf{Proof of Theorem \ref{main-thm-4}}\quad
If necessary, we add more vertices and edges such that the set of all different vertices in $M$
can be listed as $\{z_1,z_2,\ldots,z_{2m-1},z_{2m}\}$, such that for $1\leq i\leq m$, $L_{z_{2i-1},z_{2i}}$ is an edge. We also list all the Minkowski triangles in $M$ as $\{T_1,\ldots,T_n\}$ and denote the vertices of  $T_j$ as $z_{j,1}$, $z_{j,2}$ and $z_{j,3}$. It is easily known that the total number of the  edges is $3n/2$, which implies that $n$ is an even number, and $\chi(M)=2m-\frac12n$.

By Lemma \ref{lemma-4}, for each $1\leq i\leq m$, we have
\begin{equation}\label{0009}
K(z_{2i-1})+K(z_{2i})=
2\theta_M-l_{z_{2i-1}}^\pm(\mathcal{S}_{z_{2i-1}}^\pm M)
-l_{z_{2i}}^\pm(\mathcal{S}_{z_{2i}}^\pm M).
\end{equation}
Taking the summation  of (\ref{0009}) for all $i$, we get
\begin{eqnarray*}
\sum_{x\in M}K(x)&=&
2m\theta_M-\sum_{i=1}^{2m}l_{z_i}^\pm(\mathcal{S}_{z_i}^\pm M)\\
&=&2m\theta_M-\frac12\sum_{i=1}^{2m}
(l_{z_i}^+(\mathcal{S}_{z_i}^+M)+
l_{z_i}^-(\mathcal{S}_{z_i}^-M))
\end{eqnarray*}
The angles in $\sum_{i=1}^{2m}
(l_{z_i}^+(\mathcal{S}_{z_i}^+M)+
l_{z_i}^-(\mathcal{S}_{z_i}^-M))$ can be calculated triangle by triangle, i.e.
\begin{eqnarray*}
\sum_{x\in M}K(x)
&=&2m\theta_M-\frac12\sum_{j=1}^n
(\sphericalangle_{z_{j,1}}^+(L_{z_{j,2},z_{j,3}})+
\sphericalangle_{z_{j,1}}^-(L_{z_{j,2},z_{j,3}})+
\sphericalangle_{z_{j,2}}^+(L_{z_{j,1},z_{j,3}})\\ & &+
\sphericalangle_{z_{j,2}}^-(L_{z_{j,1},z_{j,3}})+
\sphericalangle_{z_{j,3}}^+(L_{z_{j,1},z_{j,2}})+
\sphericalangle_{z_{j,3}}^-(L_{z_{j,1},z_{j,2}})\\
&=&\theta_M(2m-\frac12 n)=\theta_M\chi(M).
\end{eqnarray*}
This completes the proof of Theorem \ref{main-thm-4}.

{\bf Acknowledgement.} We are grateful to  the referee for  valuable comments and suggestions. We are indebted to Professor  Fuquan Fang for providing us with  some references about piecewise flat geometry,   and Huibin Chang for some inspiring discussions.

\end{document}